\documentclass[11pt]{article}
\usepackage{amsmath,latexsym,amssymb,amsfonts,amsbsy, amsthm, mathrsfs}
\def\inte#1{
\displaystyle\mathop{#1\kern0pt}^\circ }

\let\pa=\partial

\let\d=\delta
\let\e=\varepsilon

\let\f=\frac

\let\p=\psi

\let\D=\Delta

\let\Om=\Omega
\let\wt=\widetilde

\def\cB{{\mathcal B}}
\def\cC{{\mathcal C}}

\def\cR{{\mathcal R}}
\def\cS{{\mathcal S}}

\def\pa{\partial}
\def\grad{\nabla}

\def\v{{\rm v}}

\def\virgp{\raise 2pt\hbox{,}}
\def\cdotpv{\raise 2pt\hbox{;}}

\def\eqdefa{\buildrel\hbox{\footnotesize def}\over =}

\def\C{\mathop{\mathbb C\kern 0pt}\nolimits}
\def\DD{\mathop{\mathbb D\kern 0pt}\nolimits}
\def\EE{\mathop{{\mathbb E \kern 0pt}}\nolimits}
\def\K{\mathop{\mathbb K\kern 0pt}\nolimits}
\def\N{\mathop{\mathbb N\kern 0pt}\nolimits}
\def\Q{\mathop{\mathbb Q\kern 0pt}\nolimits}
\def\R{\mathop{\mathbb R\kern 0pt}\nolimits}
\def\SS{\mathop{\mathbb S\kern 0pt}\nolimits}
\def\ZZ{\mathop{\mathbb Z\kern 0pt}\nolimits}
\def\TT{\mathop{\mathbb T\kern 0pt}\nolimits}
\def\P{\mathop{\mathbb P\kern 0pt}\nolimits}

\newcommand{\la}{\lambda}

\newcommand{\Z}{{\ZZ}}

\def\dv{\mbox{div}}

\def\dive{\mathop{\rm div}\nolimits}

\def\Supp{\mathop{\rm Supp}\nolimits\ }

\def\no{\noindent}
\def\na{\nabla}
\def\p{\partial}

\newcommand{\beq}{\begin{equation}}
\newcommand{\eeq}{\end{equation}}
\newcommand{\ben}{\begin{eqnarray}}
\newcommand{\een}{\end{eqnarray}}
\newcommand{\beno}{\begin{eqnarray*}}
\newcommand{\eeno}{\end{eqnarray*}}
\newcommand{\andf}{\quad\hbox{and}\quad}

\newtheorem{defi}{Definition}[section]
\newtheorem{thm}{Theorem}[section]
\newtheorem{lem}{Lemma}[section]
\newtheorem{rmk}{Remark}[section]

\newtheorem{prop}{Proposition}[section]
\renewcommand{\theequation}{\thesection.\arabic{equation}}

\begin{document}
\title{Global well-posedness for the $2$-D  inhomogeneous
incompressible Navier-Stokes system with large initial data in critical spaces}

\author{ Hammadi Abidi \footnote{D\'epartement de Math\'ematiques, Facult\'e des Sciences de Tunis Universit\'e de Tunis El Manar, 2092 Tunis, Tunisia . Email: {\tt hamadi.abidi@fst.rnu.tn}.} \and
Guilong Gui \footnote{Center for Nonlinear Studies, School of Mathematics, Northwest University, Xi'an 710069, China. Email: {\tt glgui@amss.ac.cn}.}}

\date{}

\setcounter{equation}{0}
\maketitle
\begin{abstract}
Without any smallness assumption, we prove the global unique solvability of the $2$-D incompressible inhomogeneous Navier-Stokes equations with initial data in the critical Besov space, which is almost the energy space in the sense that they have the same scaling in terms of this $2$-D system.
\end{abstract}

\vskip 0.2cm



\renewcommand{\theequation}{\thesection.\arabic{equation}}
\setcounter{equation}{0}

\section{Introduction}

In this paper, we consider the Cauchy problem of the
following $2$-D incompressible inhomogeneous Navier-Stokes equations
\begin{equation}
 \left\{\begin{array}{l}
\displaystyle \pa_t \rho + \dv (\rho u)=0,\qquad (t,x)\in\R^+\times\R^d, \\
\displaystyle \pa_t (\rho u) + \dv (\rho u\otimes u) -\dv (\mu \mathbb{D}(u))+\grad\Pi=0, \\
\displaystyle \dv\, u = 0, \\
\displaystyle \rho|_{t=0}=\rho_0,\quad \rho u|_{t=0}=m_0,
\end{array}\right. \label{1.1general}
\end{equation}
where $\rho, u=(u_1,u_2,...,u_d)$ stand for the density and  velocity
of the fluid respectively, $d=2,\,3$,  $\mathbb{D}(u) =\frac{1}{2}(\pa_{i}u_{j}+\pa_{j} u_{i}),$ $\Pi$  is a scalar
pressure function, and in general, the viscous coefficient $\mu(\rho)$ is a smooth, positive function on
$[0,\infty).$ Such system describes a fluid which is obtained by
mixing two miscible fluids that are incompressible and that have
different densities. It may also describe a fluid containing a
melted substance. One may check \cite{LP} for the detailed
derivation.

When $\mu(\rho)$ is independent of $\rho,$ that is, $\mu$ is a positive
constant (taking $\mu=1$ for simplicity), and $\rho_0$ is bounded away from $0$,
the system is rewritten as the form
\begin{equation}
\left\{\begin{array}{l}
\displaystyle \pa_t \rho + \dv (\rho u)=0,\qquad (t,x)\in\R_+\times\R^d, \\
\displaystyle \rho\pa_t u +\rho(u\cdot\nabla)u -\Delta u+\grad\Pi=0, \\
\displaystyle \dv\, u = 0, \\
\displaystyle(\rho,u)|_{t=0}=(\rho_0,u_0).
\end{array}\right.
\label{1.1}
\end{equation}
Kazhikov \cite{KA}
proved the global existence of strong solutions to
the system \eqref{1.1} for small smooth data in three space dimensions and all smooth data in
two dimensions, also proved that the $d$-dimensional system \eqref{1.1general} (with $d=2, 3$)
has at least one global weak solutions in the energy space. However, the uniqueness of both type weak solutions
has not be solved. Considering the case of the bounded domain $\Om$ with homogeneous Dirichlet boundary condition for the fluid velocity, Lady\v zenskaja and Solonnikov  \cite{LS} first
addressed the question of unique resolvability of \eqref{1.1general}. In particular, under the assumptions that $u_0\in W^{2-\frac{2}{p},p}(\Om)$ $(p>2)$ is
divergence free and vanishes on  $\p\Om$ and that $\rho_0\in C^1(\Om)$
is bounded away from zero, then they \cite{LS} proved global well-posedness of \eqref{1.1general} in dimension $d=2$. Similar results were obtained by Danchin \cite{danchin04} in $\R^2$
with initial data in the almost critical Sobolev spaces.

In general, DiPerna and Lions \cite{DL, LP} proved the global
existence of weak solutions to \eqref{1.1general} in energy space in any space
dimensions. Yet the uniqueness and regularities of such weak
solutions are big open questions  even in two space dimension, as
was mentioned by Lions in \cite{LP}.

On the other hand, if the density $\rho$ is away from zero, we
denote by $a\eqdefa{\rho^{-1}}-1$, then the system \eqref{1.1} can be
 equivalently reformulated as
\begin{equation}
\quad\left\{\begin{array}{l}
\displaystyle \pa_t a + u \cdot \grad a=0,\qquad (t,x)\in \R^+\times\R^d,\\
\displaystyle \pa_t u + u \cdot \grad u+ (1+a)(\grad\Pi-\Delta\,u)=0, \\
\displaystyle \dv\, u = 0, \\
\displaystyle (a, u)|_{t=0}=(a_0, u_{0}).
\end{array}\right.\label{1.1-aform}
\end{equation}
Just as the classical Navier-Stokes system, which is the case when
$a=0$ in \eqref{1.1-aform}, the system \eqref{1.1-aform} also has a scaling-invariant transformation. Indeed if $(a,
u)$ solves \eqref{1.1-aform} with initial data $(a_0, u_0)$, then for $\forall
\, \ell>0$,
\begin{equation}\label{scaling-inv}
(a, u)_{\ell}(t, x) \eqdefa (a(\ell^2\cdot, \ell\cdot), \ell
u(\ell^2 \cdot, \ell\cdot))
\end{equation}
is also a solution of \eqref{1.1-aform} with initial data $(a_0(\ell\cdot),\ell
u_0(\ell\cdot))$. Some results about global existence and uniqueness of the solutions in critical spaces for small data were proved in \cite{A,A-P,DM1}. Recently, we \cite{A-G-Z-2} first investigated
the well-posedness of the 3-D incompressible inhomogeneous Navier-Stokes equation \eqref{1.1} with initial data $(a_0,u_0)$ in the critical spaces and without size restriction on $a_0.$

For the two-dimensional case, when the density and the velocity have more regularity, R. Danchin \cite{danchin04} proved the global well-poedness result of the system \eqref{1.1}. More precisely, if $0<m\le\rho_0\le M,$
$\frac{1}{\rho_0}-1\in H^{1+\alpha}$ and $u_0\in H^\beta$ with $\alpha,\beta>0$,  the system \eqref{1.1} is globally well-posed. Recently, some improvements of this result have been achieved. Paicu, Zhang, and Zhang \cite{PZZ} investigated the unique solvability  of the global solution of the 2-D system \eqref{1.1} if $0<m\le\rho_0\le M$ and $u_0\in H^s$ with $s>0$, and the first author in the paper and Zhang \cite{A-Z} proved the global existence and uniqueness of the solution to the 2-D system \eqref{1.1} if $0<m\le\rho_0\le M,$
$\frac{1}{\rho_0}-1\in\dot B^1_{2,1}\cap\dot B^{\alpha}_{\infty,\infty}$ with $\alpha>0$
and $u_0\in\dot B^0_{2,1}$, and Danchin and Mucha \cite{DM} studied the existence and uniqueness of global solution to \eqref{1.1} if $0\le\rho_0\le M,$ $\int_{\R^2}\rho_0>0$ and $u_0\in H^1.$

In summary, all the well-posedness results of the $2$-D system \eqref{1.1} obtained so far are under the additional assumption that the density or the velocity has more regularity compared to the critical spaces.

In this paper, we investigate the global well-posedness of the $2$-D  inhomogeneous
incompressible Navier-Stokes system \eqref{1.1} with large initial data in the critical space, which is {\bf almost} the energy space in the sense that they have the same scaling in terms of the system \eqref{1.1general} (see Remark \ref{rmkthm1.1}).

The main theorem of the paper is stated as follows.
\begin{thm}\label{thm1.1}
\par Let  $m, M$ be two positive constants and $\varepsilon\in (0,1).$
Let $u_0\in\dot B^{0}_{2,1}(\R^2)$ be a
solenoidal vector field and $\frac{1}{\rho_0}-1\in\dot B^{\varepsilon}_{\frac{2}{\varepsilon},1}(\R^2)$
satisfy
\beq \label{thma.1}
m\leq\rho_0\leq M.
\eeq
Then  the system \eqref{1.1} has a  global  solution $(\rho,u,\na\Pi)$  with
\begin{equation}\label{AG}
\begin{aligned}
&\frac{1}{\rho}-1 \in C(\R_+;\,\dot B^{\varepsilon}_{\frac{2}{\varepsilon},1}(\R^2))
\\&
u\in C(\R_+;\,\dot B^{0}_{2,1}(\R^2))
\cap L^1_{loc}(\R_+;\,\dot B^{2}_{2,1}(\R^2))\andf
\\&
\p_tu,\,
\na\Pi\in L^1_{loc}(\R_+;\,\dot B^{0}_{2,1}(\R^2)).
\end{aligned}
\end{equation}
Furthermore, if, in addition, $\frac{1}{\rho_0}-1\in B^1_{2,1},$
then the solution is unique.
\end{thm}

\begin{rmk}\label{rmkthm1.1}
\par Compared to the theorem of global weak solutions in the energy space (\cite{KA,DL, LP}), where $\rho_0-1 \in L^\infty(\mathbb{R}^2)$ and $u_0 \in L^2(\mathbb{R}^2)$, especially in the non-vacuum case, in the assumptions of Theorem \ref{thm1.1}, the initial density $\rho_0 -1\in
B^1_{2,1}(\R^2))$ has the same scaling as $\rho_0-1 \in L^\infty(\mathbb{R}^2)$ in terms of the scaling-invariant transformation \eqref{scaling-inv} of the system \eqref{1.1}, and the initial velocity $u_0\in\dot B^{0}_{2,1}(\R^2)$  has the same scaling and regularity as $u_0 \in L^2(\mathbb{R}^2)$ in the energy space.
\end{rmk}

The proof of Theorem \ref{thm1.1} is completed in Sections 2-4. We now present a summary of the principal
difficulties we encounter in our analysis as well as a sketch of the key ideas used in our proof.

The first difficulty to the proof of Theorem \ref{thm1.1} lies in the fact
that when $a$ is not small, we can not use the classical arguments
in \cite{A, A-P} to deal with the following
linearized momentum equations of \eqref{1.1-aform}:
\begin{equation}\label{1.1-mom-linear-1}
\pa_t u -
(1+a)( \Delta u-\grad\Pi)=f,
\end{equation}
Motivated by \cite{DAN} and \cite{A-G-Z-2}, for some large enough integer $m,$ we
shall rewrite \eqref{1.1-mom-linear-1} as
\begin{equation}\label{1.1-mom-linear-2}
\pa_t u -
(1+ \dot{S}_ma)( \Delta u-\grad\Pi)=(a-\dot{S}_ma)( \Delta u-\grad\Pi)+f,
\end{equation}
with $\dot{S}_ma$ being partial sum of $a$ defined in \eqref{LP-decom-sum-1} in Appendix.
Then the basic energy method can be used to solve \eqref{1.1-mom-linear-2} when we deal with the global existence of the solution to \eqref{1.1-aform}.

The other difficulty in the proof of Theorem \ref{thm1.1} is how to deal with the uniqueness issue of the solution. In order to solve this problem, the crucial part is, roughly speaking, to control the $Lip$ norm of the velocity $u$, which will conserve all the regularities of the density and the velocity in the critical spaces, as well as the smallness of $a-\dot{S}_ma$ with $m$ being large enough.

 Usually, if the density or the velocity has more regularity than the critical space, the losing estimates for transport equations and the theory of transport-diffusion equations \cite{BCD} provide the boundness of the $Lip$ norm of the velocity, which will in turn close the estimates in the proof of global well-posedness of \eqref{1.1} (see \cite{PZZ, A-Z, DM}).

 In our critical case, there is no more regularity of the density or the velocity to rescue their losing regularity when we solve the transport equation of the density or the transport-diffusion equations in terms of the velocity.

For this reason, we need first to get, at least in a small time interval, the $L^1([0,T];\dot B^2_{2,1})$ estimate for the velocity field (see Proposition \ref{lema.1.1}), which relies on more elaborate application of Littlewood-Paley theory, as well as the basic energy and the estimate of $\|\nabla\Pi\|_{L^1_t(L^2)}$. Based on this, together with Osgood' lemma applied, we solve the uniqueness issue of the solution to \eqref{1.1} in the critical space.

The rest of the paper is organized as follows. In Section 2, we derive some qualitative
and analytic properties of the flow, as well as the necessary commutator estimates.  We prove the $L^1([0,T];\dot B^2_{2,1})$ estimate for the velocity field in Section 3. The proof of Theorem \ref{thm1.1} is completed in Section 4. Finally, we recall some basic ingredients of Littlwood-Paley theory in Appendix.

\medbreak \noindent{\bf Notations:} Let $A, B$ be two operators, we
denote $[A, B]=AB-BA,$ the commutator between $A$ and $B$. For
$a\lesssim b$, we mean that there is a uniform constant $C,$ which
may be different on different lines, such that $a\leq Cb$ and $C_0$ denotes a positive constant depending  only on the initial data.

For $X$ a Banach space and $I$ an interval of $\R,$ we denote by
${\mathcal{C}}(I;\,X)$ the set of continuous functions on $I$ with
values in $X,$ and by ${\mathcal{C}}_b(I;\,X)$ the subset of bounded
functions of ${\mathcal{C}}(I;\,X).$ For $q\in[1,+\infty],$ the
notation $L^q(I;\,X)$ stands for the set of measurable functions on
$I$ with values in $X,$ such that $t\longmapsto\|f(t)\|_{X}$ belongs
to $L^q(I).$

\renewcommand{\theequation}{\thesection.\arabic{equation}}
\setcounter{equation}{0}
\section{Preliminaries}
In this section, we shall first derive the following
commutator's estimate which will be frequently used throughout the
succeeding sections.
\begin{lem}
\label{lemcommu}
\par Let   $\alpha\in (-1,1),$ $(p,r)\in[1,\infty]^2,$ $u\in \dot
B^{\alpha}_{p,r}(\R^2)$ and $\nabla v\in L^\infty(\R^2)$ with
$\dive v=0.$ Then there holds
\begin{equation}\label{commutator-2}
\|[\dot\Delta_{q}, v\cdot \grad ]u\|_{L^p}
\lesssim
c_{q,r}2^{-q\alpha}\|\grad v\|_{L^{\infty}}\|u\|_{\dot{B}^{\alpha}_{p, r}}.
\end{equation}
\end{lem}
\begin{proof}
Thanks to Bony's decomposition \eqref{bony} and the divergence free
condition of $v,$ we write
\begin{equation}\label{comma}
[\dot\Delta_q,v\cdot\nabla] u=\dot\Delta_q\bigl(\partial_j
\cR(u,v^j))+\dot\Delta_q\bigl(T_{\partial_j u}v^j\bigr)
-R(v^j,\dot\Delta_q\partial_j u)-[T_{v^j},\dot\Delta_q]\partial_j u\eqdefa
\sum_{i=1}^4 \mathcal{R}^i_q.
\end{equation}
For $\mathcal{R}^1_q=\sum_{k\geq q-3}\dot\Delta_q\partial_{j}(\dot\Delta_ku\widetilde{\dot\Delta}_kv^j)$, it follows from Lemma \ref{lem2.1} and the condition $\alpha >-1$ that
\beno \begin{split}
&\|\mathcal{R}^1_q\|_{L^p}\lesssim 2^{q}\sum_{k\geq q-3} \|\dot\Delta_ku\widetilde{\dot\Delta}_kv^j \|_{L^p}\lesssim 2^{q}\sum_{k\geq q-3} \|\dot\Delta_ku \|_{L^p}\|\widetilde{\dot\Delta}_kv^j \|_{L^\infty}\\
&
\lesssim 2^{q}\|\nabla\,v\|_{L^\infty}\sum_{k\geq q-3} \|\dot\Delta_ku \|_{L^p}2^{-k}\lesssim
c_{q,r}2^{-q\alpha}\|\grad v\|_{L^{\infty}}\|u\|_{\dot{B}^{\alpha}_{p, r}}.
\end{split}
\eeno
Notice that for $\alpha<1$,
\beno \begin{split}
&\|\dot S_{k-1}\nabla u\|_{L^p}\lesssim \sum_{\ell\le
k-2}2^{\ell}\|\dot\Delta_{\ell}u\|_{L^p}
\lesssim
c_{k,r}2^{k(1-\alpha)}\| u\|_{\dot B^{\alpha}_{p,r}},
\end{split}
\eeno
then for $ \mathcal{R}^2_q=\dot\Delta_q\bigl(T_{\partial_j u}v^j\bigr)=\sum_{\vert
q-k\vert\leq 4} \dot\Delta_q(\dot S_{k-1}\partial_j
u\dot\Delta_kv^j)$,
we get
\beno
\begin{split}
\nonumber\Vert \mathcal{R}^2_q\Vert_{L^p}
\lesssim
c_{q,r}2^{-q\alpha}\|\grad v\|_{L^{\infty}}\|u\|_{\dot{B}^{\alpha}_{p, r}}.
\end{split}
\eeno Whereas thanks to the properties to the support of Fourier
transform to $\dot S_{k+2}\dot\Delta_q\partial_j u,$ one has
$$
\mathcal{R}^3_q=-R(v^j,\dot\Delta_q\partial_j u)=-\sum_{k\geq
q-2}\dot S_{k+2}\dot\Delta_q\partial_j u\dot\Delta_k v^j,
$$
which implies that for $\alpha\in \mathbb{R}$
\beno \begin{split}
\Vert \mathcal{R}^3_q\Vert_{L^p}
&\lesssim
\Vert\dot\Delta_q u\Vert_{L^p} \sum_{k\geq q-2}2^{q-k}\,
\Vert\dot\Delta_k\nabla v\Vert_{L^\infty}\\
& \lesssim
c_{q,r}2^{-q\alpha}\Vert u\Vert_{\dot B^\alpha_{p,r}}\Vert\nabla
v\Vert_{L^\infty}.
\end{split}
\eeno
For the last term in \eqref{comma}, we write
$$
\mathcal{R}^4_q=[\dot\Delta_{q},T_{v^j}]\partial_j u= \sum_{\vert
k-q\vert\leq 4}[\dot\Delta_{q},S_{k-1}v^j]\dot\Delta_k\partial_j u,
$$
which along with the classical estimate (see \cite{BCD} for example)
$$
\begin{aligned}
\Vert[\dot S_{k-1}u^j,\dot\Delta_q]\dot\Delta_k\partial_j u\Vert_{L^p}
&\lesssim
2^{-q}
\|\nabla\dot S_{k-1}v\|_{L^\infty}\|\partial_{j}\dot\Delta_{k}u\|_{L^p}
\\&
\lesssim 2^{k-q}\Vert\nabla v\Vert_{L^\infty}\Vert\dot\Delta_k
u\Vert_{L^p},
\end{aligned}
$$
 yields that for $\alpha\in \mathbb{R}$
$$
\Vert \mathcal{R}^4_q\Vert_{L^p}
\lesssim
c_{q,r}2^{-q\alpha}\Vert u\Vert_{\dot B^\alpha_{p,r}}\Vert\nabla v\Vert_{L^\infty}.
$$
This achieves the proof of Lemma \ref{lemcommu}.
\end{proof}
Applying Lemma \ref{lemcommu} to the transport equation, we have
\begin{prop}\label{small-density-1}
\par Let $k \in \Z$, $\alpha\in[0,1),$ $(p,r)\in[1,\infty]^2,$ $a_0 \in \dot B^{\alpha}_{p,r}(\R^2)$, $\nabla u \in
L^1_T(L^\infty)$ with $\dive \, u=0$, and the function $a \in
{\mathcal{C}}([0,{T}];\, \dot B^{\alpha}_{p, r}(\R^2)))$ solves
\begin{equation}\label{equation-T}
\begin{cases}
&\partial_ta+u\cdot\nabla a=0,\quad (t, x) \in [0,{T}]\times \mathbb{R}^2,\\
&a|_{t=0}=a_0 \quad x\in \mathbb{R}^2.
\end{cases}
\end{equation}
 Then there holds that for $\forall \,
t \in (0, T]$
\begin{equation}\label{dens}
\|a-\dot S_ka\|_{\widetilde{L}^\infty_t(\dot B^{\alpha}_{p,r})} \leq
\Bigl(\sum_{q\geq k}2^{rq\alpha}\|\dot\Delta_q a_0\|_{L^p}^r\Bigr)^{\f1r} +\|a_0\|_{\dot
B^{\alpha}_{p,r}} (e^{CU(t)}-1)
\end{equation}
with $U(t)=\|\nabla u\|_{L^1_t(L^\infty)}$.
\end{prop}
\begin{proof}
Applying $\dot \Delta_q$ to \eqref{equation-T} yields
\begin{equation}\label{equation-T-1}
\pa_t \dot \Delta_q a+u \cdot \grad \dot \Delta_q a=[u \cdot
\grad,\dot \Delta_q]a
\end{equation}
From the maximum principle we deduce
\begin{equation}\label{equation-T-2}
\|\dot \Delta_q a(t)\|_{L^p} \leq  \|\dot \Delta_q a_0\|_{L^p}+
\int_0^t \|[u \cdot \grad, \dot \Delta_q]a\|_{L^p}(\tau) \, d\tau.
\end{equation}
While thanks to \eqref{commutator-2}, we have
\begin{equation*}\label{est-22}
\|[u \cdot \grad, \dot \Delta_q]a\|_{L^p} \lesssim c_{q,r} 2^{-q\alpha}
\|\nabla u\|_{L^\infty} \|a\|_{\dot B^{\alpha}_{p,r}},
\end{equation*}
which along with \eqref{equation-T-2} implies that
\begin{equation*}
\|a(t)\|_{\dot B^{\alpha}_{p,r}}
\leq
\|a_0\|_{\dot B^{\alpha}_{p,r}}+C \int_0^t\|\nabla u\|_{L^1_{\tau}(L^\infty)} \|a(\tau)\|_{\dot B^{\alpha}_{p,r}}\,d\tau.
\end{equation*}
Hence, it follows from Gronwall's inequality that
\begin{equation}\label{density-est-1}
\|a(t)\|_{\dot B^{\alpha}_{p,r}}
\leq
C\|a_0\|_{\dot B^{\alpha}_{p,r}}e^{C\|\nabla u\|_{L^1_t(L^\infty)}},
\end{equation}
from which, we use \eqref{equation-T-2} again to deduce that
\begin{equation}\label{equation-T-3}
2^{q\alpha}\|\dot \Delta_q a(t)\|_{L^{\infty}_t (L^p)} \leq
2^{q\alpha}\|\dot \Delta_q a_0\|_{L^p}+C\|a_0\|_{\dot B^{\alpha}_{p,r}}
\int_0^t c_{q,r}(\tau)\|\nabla u(\tau)\|_{L^\infty} e^{C\|\nabla u\|_{L^1_{\tau}(L^\infty)}}\, d\tau.
\end{equation}
Then, taking the $\ell^r$ norm in terms of $q\in\{q \geq k\}$  leads to
\begin{equation*}
\|a-\dot S_ka\|_{\widetilde{L}^\infty_t(\dot B^{\alpha}_{p,r})}
\leq
\Bigl(\sum_{q\geq k}2^{rq\alpha}\|\dot\Delta_q a_0\|_{L^p}^r\Bigr)^{\f1r} +\|a_0\|_{\dot B^{\alpha}_{p,r}}
\int_0^t C\,U'(\tau) e^{C\,U(\tau)}\, d\tau,
\end{equation*}
which implies \eqref{dens}.
\end{proof}

\noindent To prove the uniqueness part of Theorem \ref{thm1.1}, we need the following Propositions.
\begin{prop}\label{prop-uniqueness-1}
\par Let $u_0\in B^{-1}_{2,\infty}(\R^2)$ and $v$ be a divergence free
vector field satisfying $v\in L^1_T(B^1_{\infty,1})$. Let
$f\in\widetilde L^1_T( B^{-1}_{2,\infty}),$
and
$a \in\widetilde L^\infty_T( B^{2}_{2,1})$ with $
1+a\geq c_1>0,$ we assume that
$u \in L^{\infty}_T( B^{-1}_{2,\infty})
\cap\widetilde L^1_T( B^{1}_{2, \infty})$ and $\grad\Pi\in
\widetilde L^1_T( B^{-1}_{2,\infty}),$ which solves
\begin{equation}\label{model-unique-1}
\begin{cases}
\displaystyle\partial_t u +v \cdot \grad u -(1+a)( \Delta u-\nabla \Pi) = f,\quad (t, x) \in \R_+\times\R^2,\\
\displaystyle{\mathop{\rm div}}\, u=0,\\
\displaystyle u_{|t=0}=u_0.
\end{cases}
\end{equation}
Then there holds:
\begin{equation}\label{estimate-uniqueness-1}
\begin{aligned}
\| u\|_{L^\infty_T(B^{-1}_{2,\infty})}
+\|u\|_{\widetilde L^1_T(B^{1}_{2, \infty})}\leq&
Ce^{C\bigl(T+T\|a\|_{\widetilde L^\infty_T(B^2_{2,1})}^2+\|v\|_{L^1_T(B^1_{\infty,1})}\bigr)}
\\&\quad\times\Bigl\{\|u_0\|_{B^{-1}_{2,\infty}}
+\|f\|_{\widetilde L^1_T(B^{-1}_{2,\infty})}
+\|a\|_{\widetilde L^\infty_T(B^{1}_{2,1})}
\|\nabla\Pi\|_{\widetilde L^1_T(B^{-1}_{2,\infty})}\Bigr\}.
\end{aligned}
\end{equation}
\end{prop}
\begin{proof}
Applying  $\Delta_q\mathbb{P}$ to \eqref{model-unique-1}, then a standard
commutator process gives
\begin{equation}\label{model-1-unique-1}
\begin{split}
\partial_t\Delta_q u +(v\cdot\nabla)\Delta_qu &-{\mathop{\rm
div}}((1+a)\nabla\Delta_q u)\\
&= [v, \Delta_q\mathbb{P}]\cdot\grad u-\Delta_q\mathbb{P}(\nabla a\cdot\nabla u)
+\Delta_q\mathbb{P} (T_{\nabla a}\Pi)-\Delta_q\mathbb{P} T_{\grad \Pi} a\\
&-\Delta_q\mathbb{P} \mathcal{R}(\grad \Pi, a)+{\mathop{\rm
div}}\big[\Delta_q\mathbb{P}, a\big]\nabla u+\Delta_q\mathbb{P} f.
\end{split}
\end{equation}
Thanks to the fact that
$\dive u=\dive v=0$ and $1+a\geq c_1,$ we get by taking the $L^2$
inner product of \eqref{model-1-unique-1} with $\Delta_q u$ that
\begin{equation*}
\begin{split}
&\frac{d}{dt} \|\Delta_q u\|_{L^{2}}^2
-\int_{\R^2}{\mathop{\rm
div}}((1+a)\nabla\Delta_q u)\Delta_q udx
 \\&
 \lesssim
 \|\Delta_qu\|_{L^2}\Big(
 \|[v,\Delta_q\mathbb{P}]\cdot \grad u\|_{L^2}
+\|\Delta_q\mathbb{P}(\nabla a\cdot\nabla u)\|_{L^2}
+2^{q}\|\big[a,\Delta_q\mathbb{P}\big]\nabla u\|_{L^2}
\\&\qquad\qquad\qquad+\|\Delta_q\mathbb{P} (T_{\nabla a}\Pi)\|_{L^2}
+\|\Delta_q\mathbb{P} T_{\grad \Pi} a\|_{L^2}
+\|\Delta_q\mathbb{P}\mathcal{R}(\grad \Pi, a)\|_{L^2}
+\|\Delta_q\mathbb{P} f\|_{L^2}\Big).
\end{split}
\end{equation*}
We  get, by using integration by parts and  Lemma A.5 of
\cite{DAN}, that for $0\le q$
\beno \begin{split}
-\int_{\R^2}{\mathop{\rm
div}}((1+a)\nabla\Delta_q u)\Delta_q udx =&\int_{\R^2}(1+a)|\D_q\na
u|^2\,dx
\gtrsim
2^{2q}\|\Delta_qu\|_{L^2}^2,
\end{split}
\eeno
This leads to
\begin{equation}\label{model-unique-1a}
\begin{aligned}
&\|\Delta_q u\|_{L^\infty_t(L^2)}
+2^{2q}\|\Delta_qu\|_{L^1_t(L^2)}
\lesssim
\|\D_qu_0\|_{L^2}
+\|\Delta_{-1}u\|_{L^1_t(L^2)}
\\&
+\int_0^t\bigl(\|[v, \Delta_q\mathbb{P}]\cdot\grad u\|_{L^2}
+\|\Delta_q\mathbb{P}(\nabla a\cdot\nabla u)\|_{L^2}
+2^{q}\|\big[\Delta_q\mathbb{P},a\big]\nabla u\|_{L^2}\\
&\qquad
+\|\Delta_q\mathbb{P} (T_{\nabla a}\Pi)\|_{L^2}
+\|\Delta_q\mathbb{P} T_{\grad\Pi} a\|_{L^2}
+\|\Delta_q\mathbb{P} \mathcal{R}(\grad \Pi,a)\|_{L^2}
+\|\Delta_q\mathbb{P} f\|_{L^2} \bigr)\,dt',
\end{aligned}
\end{equation}
First as $\dv\,v=0,$ we deduce by
similar proof to inequality \eqref{a.4}
\begin{equation}\label{nonhomo-commutator-2}
\sup_{q\geq-1}2^{-q}\bigl\|[v,\Delta_q\mathbb{P}]\cdot\grad u\bigr\|_{L^1_T(L^2)}
\lesssim
\int_0^T\| v\|_{B^1_{\infty,1}}\|u\|_{B^{-1}_{2,\infty}}\,dt.
\end{equation}
Thanks to \eqref{bony}  in the inhomogeneous context, we write
\begin{equation*}
\big[a,\Delta_q\mathbb{P}\big]\nabla u=\Delta_q\mathbb{P} \,\mathcal{R}(a,\nabla
u)+\Delta_q\mathbb{P}\,T_{\nabla u}a-\,T'_{\Delta_q\nabla u}a-
[T_a,\Delta_q\mathbb{P}]\nabla u.
\end{equation*}
Whereas applying Lemma \ref{lem2.1} gives
\begin{equation}\label{model-1-unique-2}
\begin{split}
\|\Delta_q\mathbb{P} \mathcal{R}(a,\nabla u)\|_{L^1_T(L^2)}
\lesssim &
\sum_{k\geq q-3} \|\Delta_k a
\|_{L^\infty_T(L^\infty)}\|\widetilde{\Delta}_k\nabla
u\|_{L^1_T(L^2)} \\
 \lesssim &
\sum_{k\geq q-3}2^{k}
\|\Delta_ka\|_{L^\infty_T(L^\infty)}
\|\widetilde{\Delta}_k u\|_{L^1_T(L^2)}
\\
\lesssim &
\|a\|_{L^\infty_T(B^{1}_{\infty,\infty})}
\|u\|_{L^1_T(B^{0}_{2,1})}.
\end{split}
\end{equation}
The same estimate holds for $T'_{\Delta_q\nabla u}a.$
Note that
\[
 \|S_{k-1}\na u\|_{L^2}\lesssim
2^{k}
\sum_{\ell\le k-2}2^{(\ell-k)}
\|\Delta_{\ell}u\|_{L^2},
\]
this along with Lemma \ref{lem2.1} leads to
\begin{equation}\label{model-1-unique-3}
\begin{split}
 \|\Delta_q T_{\nabla u}a(t)\|_{L^1_T(L^2)} &
\lesssim
\sum_{| q-k|\leq 4}\|\Delta_k a\|_{L^\infty_T(L^\infty)}
\|S_{k-1} \grad u\|_{L^1_T(L^2)}
\\
&\lesssim
\|a\|_{L^\infty_T(B^{1}_{\infty,\infty})}
\|u\|_{\widetilde L^1_T(B^{0}_{2,\infty})}.
\end{split}
\end{equation}
Finally notice that
$$
\|\nabla S_{k-1}a\|_{L^\infty}
\lesssim
\|\nabla a\|_{L^\infty},
$$
one has
\begin{equation}\label{model-1-unique-5}
\begin{split}
\| [\Delta_q\mathbb{P},T_a]\nabla u\|_{L^1_T(L^2)} &
\lesssim \sum_{|k-q|\leq
4}2^{-q}\|\nabla S_{k-1}a\|_{L^\infty_T(L^\infty)}
\|\nabla\Delta_{k}u\|_{L^1_T(L^2)}
\\
\lesssim &
\|\nabla a\|_{L^\infty_T(L^\infty)}
\|u\|_{\widetilde L^1_T(B^0_{2,\infty})}.
\end{split}
\end{equation}
As a consequence, we obtain
\begin{equation}\label{model-1-unique-6}
\sup_{q\geq-1}\|\big[\Delta_q\mathbb{P}, a\big]\nabla u\|_{L^1_T(L^2)}
\lesssim
\bigl(\|\nabla a\|_{L^\infty_T(L^\infty)}+\|a\|_{L^\infty_T(B^1_{\infty,\infty})}\bigr)
\|u\|_{L^1_T(B^0_{2,1})}.
\end{equation}
On the other hand, it follows from \eqref{bony} and Lemma \ref{lem2.1}, that
\begin{equation}\label{model-1-unique-7}
\begin{aligned}
\|\nabla a \cdot\nabla u\|_{\widetilde L^1_T(B^{-1}_{2,\infty})}
&\lesssim
\|\nabla a\|_{L^\infty_T(L^\infty)}
\|u\|_{\widetilde L^1_T(B^{0}_{2,\infty})}
+\|a\|_{L^\infty_T(B^2_{2,\infty})}
\|u\|_{L^1_T(B^{0}_{2,1})}
\\&
\lesssim
\|a\|_{L^\infty_T(B^2_{2,1})}
\|u\|_{L^1_T(B^0_{2,1})}.
\end{aligned}
\end{equation}
Similar, we have
\begin{equation}\label{model-1-unique-8}
\begin{aligned}
\sup_{q\geq-1}2^{-q}\bigl(
\|\Delta_q\mathbb{P} T_{\grad \Pi}a\|_{L^1_T(L^2)}
&+\|\Delta_q\mathbb{P}\mathcal{R}(a, \grad
\Pi)\|_{L^1_T(L^2)}\bigr)
\\&
\lesssim
\sup_{q\geq-1}\bigl(
\|\Delta_q\mathbb{P} T_{\grad \Pi}a\|_{L^1_T(L^{1})}
+\|\Delta_q\mathbb{P}\mathcal{R}(a, \grad
\Pi)\|_{L^1_T(L^{1})}\bigr)
\\&
\lesssim
\|a\|_{\widetilde L^\infty_T(B^{1}_{2,1})}
\|\nabla\Pi\|_{\widetilde L^1_T(B^{-1}_{2,\infty})}.
\end{aligned}
\end{equation}
And a similar argument gives the same estimate for
$\Delta_q\mathbb{P}T_{\grad a}\Pi.$

Plugging (\ref{nonhomo-commutator-2} - \ref{model-1-unique-8}) and
 into \eqref{model-unique-1a}, we arrive
at
\begin{equation*}\label{estimate-uniqueness-1-b}
\begin{split}
\| u\|_{L^\infty_T(B^{-1}_{2,\infty})} &
+\|u\|_{\widetilde L^1_T(B^{1}_{2, \infty})}
\leq
C\Bigl\{\|u_0\|_{B^{-1}_{2,\infty}}
+\int_{0}^T \| u(t)\|_{B^{-1}_{2,\infty}}
\bigl(1+\|v(t)\|_{B^1_{\infty,1}}\bigr)dt
\\&
+\|a\|_{\widetilde L^\infty_T(B^2_{2,1})}
\|u\|_{L^1_T(B^{0}_{2,1})}
+\|a\|_{\widetilde L^\infty_T(B^{1}_{2,1})}
\|\nabla\Pi\|_{\widetilde L^1_T(B^{-1}_{2,\infty})}
+\|f\|_{\widetilde L^1_T(B^{-1}_{2,\infty})}\Bigr\},
\end{split}
\end{equation*}
or by interpolation, we have
$$
\|u\|_{L^1_T(B^{0}_{2,1})}
\lesssim
\|u\|_{\widetilde L^1_T(B^{-1}_{2, \infty})}^{\f12}
\|u\|_{\widetilde L^1_T(B^{1}_{2, \infty})}^{\f12}.
$$
Then by  Young inequality Gronwall Lemma, we deduce
$$
\begin{aligned}
\| u\|_{L^\infty_T(B^{-1}_{2,\infty})}
+\|u\|_{\widetilde L^1_T(B^{1}_{2, \infty})}
\leq
Ce^{C\bigl(T+T\|a\|_{\widetilde L^\infty_T(B^2_{2,1})}^2+\|v\|_{L^1_T(B^1_{\infty,1})}\bigr)}
&\Bigl\{\|u_0\|_{B^{-1}_{2,\infty}}
+\|f\|_{\widetilde L^1_T(B^{-1}_{2,\infty})}
\\&
+\|a\|_{\widetilde L^\infty_T(B^{1}_{2,1})}
\|\nabla\Pi\|_{\widetilde L^1_T(B^{-1}_{2,\infty})}\Bigr\}.
\end{aligned}
$$
which completes the proof of this Proposition.
\end{proof}
\begin{prop}\label{prop-H-negtive}
\par Let $a\in B^{1}_{2,1}(\R^2)$ such that $ 0<{\underline{b}}\leq
1+a\leq {\bar{b}}$, and
\begin{equation}\label{est-unique-1}
\|a-S_k a\|_{B^{1}_{2,1}}\leq c
\end{equation}
for some sufficiently small positive constant $c$ and some integer
$k \in \N.$  Let $F \in B^{-1}_{2,\infty}(\R^2)$  and $\grad \Pi \eqdefa
\mathcal{H}_{b}({F})\in B^{-1}_{2,\infty}(\R^2)$ solves
\begin{equation}\label{unique-per-model}
\dive\,((1+a) \grad \Pi) =\dive\,F.
\end{equation}
Then there holds
\begin{equation}\label{estimate-unique-per-model}
\|\grad \Pi\|_{B^{-1}_{2,\infty}}
 \lesssim
(1+2^{k}\|a\|_{B^{1}_{2,1}}^2)
 \bigl(\| F\|_{B^{-2}_{2,\infty}} +\|\dv\, F\|_{B^{-2}_{2,\infty}}\bigr).
\end{equation}
\end{prop}
\begin{proof}
We first deduce from \eqref{est-unique-1} and ${\underline{b}}\leq
1+a$ that
\begin{equation}\label{lower-bdd-1}
1+S_k a=1+a+(S_k a-a) \geq{{\underline{b}}\over{2}}.
\end{equation}

Motivated by \cite{A-G-Z-2, danchin-3}, we shall use a duality argument to
prove \eqref{estimate-unique-per-model}. For the sake of simplicity,
we just prove \eqref{estimate-unique-per-model} for sufficiently
smooth function $F$. In order to make the following computation
rigorous, one has to use a density argument, which we omit here.

For this, we first estimate $\|\grad \Pi\|_{B^{1}_{2,1}}$ under the
assumption that $F \in B^{1}_{2,1}(\R^2)$. Indeed, we write
\begin{equation*}
\dv[(1+S_k a)\nabla\Pi]=\dv\,F+\dv\,[(S_k a-a)\nabla\Pi],
\end{equation*}
applying  $\Delta_q$ to the above equation gives
\begin{equation*}
\dv\,[(1+S_k a)\Delta_q\nabla\Pi]=\dv\,\Delta_q F+\dv\,\Delta_q
[(S_k a-a)\nabla\Pi]+\dv\,([S_k a,\Delta_q]\nabla\Pi).
\end{equation*}
Taking the $L^{2}$ inner product of this equation with $\Delta_q\Pi,$
we obtain by a similar estimated of inequality \eqref{model-1-unique-6}
\begin{equation*}\label{estimate-unique-per-model-1}
\begin{split}
 \|\grad \Pi\|_{B^{1}_{2,1}}
 &\lesssim
 \|(S_k a-a)\nabla\Pi\|_{B^{1}_{2,1}}
+\|F\|_{B^{1}_{2,1}}
+\sum_{q\geq-1}2^{q}\bigl\|[S_k a,\Delta_q]\nabla\Pi\bigr\|_{L^{2}}
 \\
&\lesssim
\|S_k a-a\|_{B^{1}_{2,1}}
\|\nabla\Pi\|_{B^{1}_{2,1}}
+\|F\|_{B^{1}_{2,1}}
+\|S_ka\|_{B^{\f12}_{\infty,1}}
\|\nabla\Pi\|_{B^{\f12}_{2,1}},
\end{split}
\end{equation*}
and by classical elliptic estimate, we have
$$
\|\nabla\Pi\|_{L^2}
\lesssim
\|F\|_{L^2}.
$$
This along with \eqref{est-unique-1} and interpolation, leads
to
\begin{equation}\label{estimate-unique-per-model-2}
\|\grad \Pi\|_{B^{1}_{2,1}}
\lesssim
\bigl(1+2^{k}\|a\|_{B^{1}_{2,1}}^2\bigr)
\|F\|_{B^{1}_{2,1}}.
\end{equation}
Now we use a duality argument (see Corollary 6.2.8 \cite{BE}) to estimate
$\|\grad\Pi\|_{B^{-1}_{2,\infty}}$ in
the case when $F\in B^{-1}_{2,\infty}(\R^2).$ Notice that
\begin{equation}\label{estimate-unique-per-model-3}
\|\grad \Pi\|_{ B^{-1}_{2,\infty}}
=\sup_{\|g\|_{ B^{1}_{2,1}} \leq 1}\bigl\langle g ,
\grad \Pi \bigr\rangle =\sup_{\|g\|_{B^{1}_{2,1}} \leq 1}(-\int \Pi\, \dv\, g
\,dx),
\end{equation}
where $\bigl\langle g , \ \grad \Pi \bigr\rangle$ denotes the
duality bracket between $\cS'(\R^2)$ and $\cS(\R^2).$  Whereas
\eqref{estimate-unique-per-model-2} ensures that for any $g\in
B^{1}_{2,1}(\R^2)$
\begin{equation*}
\dv\, ((1+a) \grad h_g)=\dv\, g
\end{equation*}
has a unique solution $\grad h_{g} \in B^{1}_{2,1}(\R^2)$ satisfying
\begin{equation}\label{estimate-unique-per-model-4}
\|\grad h_g\|_{B^1_{2,1}}
\lesssim
(1+2^{k}\|a\|_{B^{1}_{2,1}}^2)
\|g\|_{B^{1}_{2,1}},
\end{equation}
which along with \eqref{estimate-unique-per-model-3}  yields
\begin{equation*}\label{estimate-unique-per-model-5}
\begin{split}
 \|\grad \Pi\|_{B^{-1}_{2,\infty}} &=
 \sup_{\|g\|_{B^{1}_{2,1}} \leq 1}-\bigl\langle \Pi, \dv\,
((1+a) \grad h_g) \bigr\rangle=\sup_{\|g\|_{B^{1}_{2,1}} \leq 1}\bigl\langle
(1+a)\grad \Pi,
 \grad h_g \bigr\rangle\\
  &=\sup_{\|g\|_{B^{1}_{2,1}} \leq 1}-\bigl\langle h_g, \dv\,
((1+a) \grad \Pi) \bigr\rangle=\sup_{\|g\|_{B^{1}_{2,1}} \leq
1}-\bigl\langle \dv\, F,
 h_g \bigr\rangle
 \\&
 =\sup_{\|g\|_{B^{1}_{2,1}} \leq1}
- \bigl\langle \dv\, F, \sum_{\ell\geq-1}\Delta_{\ell}h_g \bigr\rangle.
\end{split}
\end{equation*}
Hence, it follows that
\begin{equation}\label{estimate-unique-per-model-5}
\begin{split}
 \|\grad \Pi\|_{B^{-1}_{2,\infty}}&
=\sup_{\|g\|_{B^{1}_{2,1}} \leq1}
- \bigl\langle \dv\, F, \Delta_{-1}h_g \bigr\rangle
+\sup_{\|g\|_{B^{1}_{2,1}} \leq1}
- \bigl\langle \dv\, F, \sum_{\ell\geq0}\Delta_{\ell}h_g \bigr\rangle
 \\&
 =\sup_{\|g\|_{B^{1}_{2,1}} \leq 1}\bigl\langle F,
 \grad\Delta_{-1} h_g \bigr\rangle
 +\sup_{\|g\|_{B^{1}_{2,1}} \leq1}
- \bigl\langle \dv\, F, \sum_{\ell\geq0}\Delta_{\ell}h_g \bigr\rangle.
\end{split}
\end{equation}
Whence thanks to \eqref{estimate-unique-per-model-4}, we obtain
\begin{equation*}\label{estimate-unique-per-model-6}
\begin{aligned}
 \|\grad \Pi\|_{B^{-1}_{2,\infty}}
 &\lesssim
 \sup_{\|g\|_{B^{1}_{2,1}} \leq 1}\| F\|_{B^{-2}_{2,\infty}} \|
 \grad\Delta_{-1} h_g\|_{B^{2}_{2,1}}
 +\sup_{\|g\|_{B^{1}_{2,1}} \leq 1}\|\dv\, F\|_{B^{-2}_{2,\infty}}
 \|\sum_{\ell\geq0}\Delta_{\ell} h_g\|_{B^{2}_{2,1}}
 \\&
 \lesssim
 \sup_{\|g\|_{B^{1}_{2,1}} \leq 1}\| F\|_{B^{-2}_{2,\infty}}
 \|\grad h_g\|_{B^{1}_{2,1}}
 +\sup_{\|g\|_{B^{1}_{2,1}} \leq 1}\|\dv\, F\|_{B^{-2}_{2,\infty}}
 \|\grad h_g\|_{B^{1}_{2,1}}
\\&
 \lesssim
 (1+2^{k}\|a\|_{B^{1}_{2,1}}^2)
 \bigl(\| F\|_{B^{-2}_{2,\infty}} +\|\dv\, F\|_{B^{-2}_{2,\infty}}\bigr),
\end{aligned}
\end{equation*}
which completes the proof of this proposition.
\end{proof}

In order to get the uniqueness of the solution in the critical case in Theorem \ref{thm1.1}, we need to recall the following Osgood's lemma \cite{fleet}.
\begin{lem}[\cite{fleet}, Osgood's lemma]\label{lem-Osgood}
Let $f \geq 0$ be a measurable function,
 $\gamma$ be a locally integrable function and $\mu$ be a positive,
continuous and nondecreasing function which verifies the following condition
\begin{equation*}
\int_0^1\frac{dr}{\mu(r)}=+\infty.
\end{equation*}
Let also $a$ be a positive real number and let $f$ satisfy the inequality
\begin{equation*}
f(t) \leq a+\int_0^t\gamma(s) \mu(f(s))\, ds.
\end{equation*}
Then if $a$ is equal to zero, the function $f$ vanishes.
If $a$ is not zero, then we have
\begin{equation*}
-\mathcal{M}(f(t))+\mathcal{M}(a)\leq \int_0^t\gamma(s) \, ds\quad
\mbox{with}\quad
\mathcal{M}(x)=\int_x^1\frac{dr}{\mu(r)}.
\end{equation*}
\end{lem}

\renewcommand{\theequation}{\thesection.\arabic{equation}}
\setcounter{equation}{0}

\section{The $L^1([0,T];\dot B^2_{2,1})$ estimate for the  velocity
field} \label{sect3}

In this section, we want to get, at least in the small time interval, the $L^1([0,T];\dot B^2_{2,1})$ estimate for the velocity field, which plays a crucial role in the study of the uniqueness of the solution to \eqref{1.1general}. For this, we first investigate some {\it a priori} estimates about the basic energy and the pressure.
\begin{prop}\label{Prop1}
\par Let $\varepsilon\in(0,1),$ $a_0:=\frac{1}{\rho_0}-1\in\dot B^0_{\infty,1}\cap\dot B^{\varepsilon}_{\frac{2}{\varepsilon},\infty}(\R^2)$ and $u_0\in L^2(\R^2)$, and \eqref{thma.1} holds. Let $(\rho,u,\na\Pi)$  be a  smooth enough solution of \eqref{1.1general} on $[0,T^\ast[,$ then for any $t \in ]0, T^{\ast}[$, there hold that
\begin{equation}\label{p.0}
\|\sqrt{\rho}u\|_{L^\infty_t(L^2)}+2\|\nabla u\|_{L^2_t(L^2)}
\le
\|\sqrt{\rho_0}u_0\|_{L^2},
\end{equation}
and
\begin{equation}\label{Pi-L2-0}
\begin{split}
\|\nabla\Pi\|_{L^1_t(L^2)}
\leq &
\Bigl(\eta+M\sum_{q\geq k}\|\dot\Delta_q a_0\|_{L^\infty}+M\|a_0\|_{\dot
B^{0}_{\infty,1}}\{e^{C\|\nabla u\|_{L^1_t(L^\infty)}}-1\}\Bigr)\|\Delta u\|_{L^1_t(L^2)}
\\&+C_{\eta}\,\bigg(\sqrt{t}2^k+\sqrt{t}2^{k}e^{C\|\nabla u\|_{L^1_t(L^\infty)}}
+\|\nabla u\|_{L^2_t(L^2)}^2\Bigr)
\end{split}
\end{equation}
for any positive constant $\eta$.
\end{prop}
\begin{proof}
We first get, by using standard energy estimate to \eqref{1.1}, that
$$
\frac{1}{2}\frac{d}{dt}\|\sqrt{\rho}u\|_{L^2}^2
+\|\nabla u\|_{L^2}^2
=0.
$$
On the other hand, let $a\eqdefa\f1\rho-1,$ the system \eqref{1.1} can be equivalently reformulated as
\begin{equation}
\left\{\begin{array}{l}
\displaystyle \pa_t a + \dv (a \,u)=0,\qquad (t,x)\in\R^+\times\R^2, \\
\displaystyle \pa_t u +(u\cdot\nabla)u -(1+a)(\Delta u-\grad\Pi)
=0, \\
\displaystyle \dv\, u = 0, \\
\displaystyle (a,u_0)|_{t=0}=(a_0,u_0).
\end{array}\right. \label{1.1.1}
\end{equation}
Applying the $\dv$ operator to the momentum equation of \eqref{1.1.1} yields that
\begin{equation}\label{AAA}
\dv\bigl\{(1+a)\nabla\Pi\bigr\}=\dv(a\Delta u)-\dv\bigl\{(u\cdot\nabla)u\bigr\},
\end{equation}
for some large enough integer $k$ we shall rewrite the above equality as
$$
\begin{aligned}
\dv\bigl\{(1+a)\nabla\Pi\bigr\}
&=
\dv\bigl\{(a-\dot S_ka)\Delta u\bigr\}+\dv(\dot S_ka\Delta u)
-\dv\bigl\{(u\cdot\nabla)u\bigr\}
\\&
=\dv\bigl\{(a-\dot S_ka)\Delta u\bigr\}+T_{\Delta u}\nabla\dot S_ka+T_{\nabla\dot S_ka}\Delta u
\\&\qquad
+\dv\bigl\{\cR(\dot S_ka,\Delta u)\bigr\}
-\dv\bigl\{(u\cdot\nabla)u\bigr\}.
\end{aligned}
$$
By taking $L^2$ inner product of the above equation with $\Pi,$ we get from the fact
 $1+a=\f1\rho\geq\f1M$ that
\begin{equation}\label{Pi-L2-1}
\begin{split}
\f1M\|\nabla\Pi\|_{L^2}^2
\le
\|\nabla\Pi\|_{L^2}\Bigl(\|(a-\dot S_ka)\Delta u\|_{L^2}
&+\|T_{\Delta u}\nabla\dot S_ka\|_{\dot H^{-1}}+\|T_{\nabla\dot S_ka}\Delta u\|_{\dot H^{-1}}
\\&
+\|\cR(\dot S_ka,\Delta u)\|_{L^2}
+\|(u\cdot\nabla)u\|_{L^2}\Bigr).
\end{split}
\end{equation}
Thanks to the product law in Besov spaces, one can see
\begin{equation*}
  \begin{split}
&\|(a-\dot S_ka)\Delta u\|_{L^2}\leq \|a-\dot S_ka\|_{L^\infty}\|\Delta u\|_{L^2},\\
&\|T_{\Delta u}\nabla\dot S_ka\|_{\dot H^{-1}}+\|T_{\nabla\dot S_ka}\Delta u\|_{\dot H^{-1}}\lesssim \|\nabla\dot S_ka\|_{L^\infty}\|\Delta u\|_{\dot H^{-1}}\lesssim \|\nabla\dot S_ka\|_{L^\infty}\|\nabla u\|_{L^2},\\
&
\|\cR(\dot S_ka,\Delta u)\|_{L^2}\lesssim \|\cR(\dot S_ka,\Delta u)\|_{\dot B^{\varepsilon}_{\frac{2}{1+\varepsilon},2}}\lesssim\|\dot S_ka\|_{\dot B^{\varepsilon+1}_{\frac{2}{\varepsilon},\infty}}\|\nabla u\|_{L^2}\lesssim2^{k}\|a\|_{\dot B^{\varepsilon}_{\frac{2}{\varepsilon},\infty}}\|\nabla u\|_{L^2},\\
&
\|(u\cdot\nabla)u\|_{L^2} \lesssim \|u\|_{L^4}\|\nabla u\|_{L^4}\lesssim \|u\|_{L^2}^{\f12}\|\nabla u\|_{L^2}\|\Delta u\|_{L^2}^{\f12}.
\end{split}
\end{equation*}
from which and \eqref{Pi-L2-1}, we infer
\begin{equation}\label{Pi-L2-2}
\begin{split}
\|\nabla\Pi\|_{L^2}\leq &
M\|a-\dot S_ka\|_{L^\infty}\|\Delta u\|_{L^2}
+C\,2^k\|a\|_{L^\infty}\|\nabla u\|_{L^2}
\\&+C\,2^{k}\|a\|_{\dot B^{\varepsilon}_{\frac{2}{\varepsilon},\infty}}\|\nabla u\|_{L^2}
+C\,\|u\|_{L^2}^{\f 12}\|\nabla u\|_{L^2}\|\Delta u\|_{L^2}^{\f 12}.
\end{split}
\end{equation}
From the maximum principle and Lemma \ref{lemcommu}, we deduce
$$
\|a(t)\|_{L^\infty}\le\|a_0\|_{L^\infty}\andf\quad
\|a\|_{L^\infty_t(\dot B^{\varepsilon}_{\frac{2}{\varepsilon},\infty})}
\le
\|a_0\|_{\dot B^{\varepsilon}_{\frac{2}{\varepsilon},\infty}}e^{C\|\nabla u\|_{L^1_t(L^\infty)}}.
$$
Thus by the Young inequality, we deduce that for any positive constant $\eta$
$$
\begin{aligned}
\|\nabla\Pi\|_{L^1_t(L^2)}
&\leq
\bigl(M\|a-\dot S_ka\|_{L^\infty_t(L^\infty)}+\eta\bigr)\|\Delta u\|_{L^1_t(L^2)}
\\&
+C_{\eta}\,\sqrt{t}2^k\|a_0\|_{L^\infty}\|\nabla u\|_{L^2_t(L^2)}
+C_{\eta}\,\sqrt{t}2^{k}\|a_0\|_{\dot B^{\frac{N\varepsilon}{2}}_{\frac{2}{\varepsilon},\infty}}\|\nabla u\|_{L^2_t(L^2)}e^{C\|\nabla u\|_{L^1_t(L^\infty)}}
\\&
+C_{\eta}\,\|u\|_{L^\infty_t(L^2)}\|\nabla u\|_{L^2_t(L^2)}^2.
\end{aligned}
$$
Thanks to Proposition \ref{small-density-1}, we get \eqref{Pi-L2-0}, which completes the proof of Proposition \ref{Prop1}.
\end{proof}
With Proposition \ref{Prop1} in hand, we are in a position to prove the following proposition about the estimate $ \|u\|_{L^1_t(\dot B^{2}_{2,1})}$.
\begin{prop}\label{lema.1.1}
\par Let $\e\in(0,1),$ $u_0\in\dot B^{0}_{2,1}$
and $a_0\in\dot B^{\varepsilon}_{\frac{2}{\varepsilon},1}.$ Let $(a,u,\na\Pi)$ be a  smooth enough
solution of \eqref{1.1.1} on $[0,T^\ast[,$ then there is small positive time $T_1<T^\ast$ such that, for all $t \leq T_1$, there holds
\beq\label{a.1rf}
\|u\|_{L^1_t(\dot B^{2}_{2,1})} +\|\nabla u\|_{L^2_t(L^2)}
\lesssim
\sum_{j\in\Z}\bigl(1-e^{-ct2^{2j}}\bigr)\|\dot\Delta_ju_0\|_{L^2}+\sqrt{t}.
\eeq
\end{prop}

\begin{proof} Let $\mathbb{P}\eqdefa I+\na(-\D)^{-1}\dv$ be the Leray
projection operator. We get, by first dividing the momentum equation
of \eqref{1.1} by $\rho$ and then applying the resulting equation by
the operator $\mathbb{P},$ that
$$
\partial_t u+\mathbb{P}\bigl(u\cdot\nabla u\bigr)
-\mathbb{P}\bigl({\rho}^{-1}\{\Delta u-\nabla\Pi\}\bigr)
=0.
$$
Applying $\dot\Delta_j$ to the above equation and using a standard
commutator's process, we write
\beq\label{a.2}
\begin{aligned}
\rho\partial_t\dot\Delta_ju+\rho
u\cdot\nabla\dot\Delta_ju-\D\dot\D_ju
&=-\rho[\dot\Delta_j\mathbb{P}; u\cdot\nabla]u
+\rho[\dot\Delta_j\mathbb{P};
{\rho}^{-1}]\bigl(\Delta u-\nabla\Pi\bigr).
\end{aligned}
\eeq
Taking $L^2$ inner product of \eqref{a.2} with
$\dot\Delta_ju,$ we obtain
\beq\label{a.2da}
\begin{aligned}
\frac{1}{2}&\frac{d}{dt}\int_{\R^2}\rho|\dot\Delta_ju|^2\,dx-
\int_{\R^2}\D\dot\Delta_ju\dot\Delta_j u\,dx
\\&
\leq \|\dot\Delta_ju\|_{L^2}
\Bigl(\|\rho[\Delta_j\mathbb{P}; u\cdot\nabla]u\|_{L^2}
+\|\rho[\dot\Delta_j\mathbb{P};\frac{1}{\rho}]\bigl(\Delta u-\nabla\Pi\bigr)\|_{L^2}\Bigr).
\end{aligned}
\eeq
We  get, by using integration by parts and  Lemma \ref{lem2.1}, that
\beno \begin{split} -\int_{\R^2}\D\dot\Delta_ju\ |\
\Delta_j u\,dx =&\int_{\R^2}|\nabla\dot\D_ju|^2\,dx
\geq
\bar{c}2^{2j}\|\dot\Delta_ju\|_{L^2}^2.
\end{split}
\eeno
 Thus, we deduce from \eqref{a.2da} that
 $$
 \begin{aligned}
&\frac{d}{dt}\|\sqrt{\rho}\dot\Delta_ju\|_{L^2}^2
+2c2^{2j}\|\sqrt{\rho}\dot\Delta_ju\|_{L^2}^2\\
&\lesssim\|\sqrt{\rho}\dot\Delta_ju\|_{L^2}(
\|[\dot\Delta_j\mathbb{P}; u\cdot\nabla]u\|_{L^2}
+\|[\dot\Delta_j\mathbb{P};\frac{1}{\rho}]\bigl(\Delta u-\nabla\Pi\bigr)\|_{L^2}),
\end{aligned}
$$
where
$c=\bar{c}\bigl/M.$ This gives rise to
\begin{equation}\label{T-D}
\begin{aligned}
&\|\sqrt{\rho}\dot\Delta_ju(t)\|_{L^2}
\lesssim
e^{-c2^{2j}t}\|\sqrt{\rho_0}\dot\Delta_ju_0\|_{L^2}
\\&
+\int_0^te^{-c2^{2j}(t-t')}
\Bigl(\|[\dot\Delta_j\mathbb{P}; u\cdot\nabla]u\|_{L^2}
+\|[\dot\Delta_j\mathbb{P};\frac{1}{\rho}]\bigl(\Delta u-\nabla\Pi\bigr)\|_{L^2}\Bigr)(t')\,dt'.
\end{aligned}
\end{equation}
As a consequence,  by virtue of Definition \ref{chaleur+}, we infer
\beq\label{a.3}
\begin{aligned}
\|u\|_{L^1_t(\dot B^{2}_{2,1})}
\lesssim
& \sum_{j\in\Z}\bigl(1-e^{-ct2^{2j}}\bigr)\|\dot\Delta_ju_0\|_{L^2}
+\sum_{j\in\Z}\|[\dot\Delta_j\mathbb{P};
u\cdot\nabla]u\|_{L^1_t(L^2)}
\\&
+\sum_{j\in\Z}
\|[\dot\Delta_j\mathbb{P};\frac{1}{\rho}]\bigl(\Delta u-\nabla\Pi\bigr)\|_{L^1_tL^2}.
\end{aligned}
\eeq
In what follows, we shall deal with term by term the right-hand
side of \eqref{a.3}. Firstly applying Bony's decomposition
\eqref{bony} yields
\beno
[\dot\Delta_j\mathbb{P};u\cdot\nabla]u
=
[\dot\Delta_j\mathbb{P};T_u\cdot\nabla]u+\dot\D_j\mathbb{P}T'_{\na u}u-T'_{\na\dot\D_ju}u.
\eeno
Hence, due to Lemma 1 of \cite{Plan}, we achieve
\beno \begin{split}
\|[\dot\Delta_j\mathbb{P};
T_u\cdot\nabla]u\|_{L^2}
\lesssim&
\sum_{|j-\ell|\leq 4}
\|\na\dot  S_{\ell-1}u\|_{L^\infty}\|\dot\D_\ell u\|_{L^2}
\\
\lesssim&
\sum_{|j-\ell|\leq 4}\|\dot\D_\ell u\|_{L^2}
 \sum_{-1\le k\le\ell-2}
 2^{k}\| \nabla\dot\Delta_ku\|_{L^2},
\end{split}
\eeno
which follows that
$$
\sum_{j\in\Z}\|[\dot\Delta_j\mathbb{P};
T_u\cdot\nabla]u\|_{L^2}
\lesssim
\|\nabla u\|_{L^2}^2.
$$
While by using Lemma \ref{lem2.1}, one has
\beno
\begin{split}
 \|\dot\D_j\mathbb{P}T'_{\na u}u\|_{L^2}
\lesssim
 2^{j}\sum_{\ell\geq j-3}\|\dot\D_\ell u\|_{L^2}\|\dot S_{\ell+2}\na u\|_{L^2}
\lesssim
\|\nabla u\|_{L^2}\|u\|_{\dot B^{1}_{2,1}},
\end{split}
\eeno
which along with the interpolation inequality
$
\|u\|_{\dot B^{1}_{2,1}}\lesssim
\|u\|_{L^2}^{\f12}\|u\|_{\dot B^2_{2,1}}^{\f12},
$
implies that
\beno
\begin{split}
\sum_{j\in\Z}\|\dot\D_j\mathbb{P}T'_{\na u}u\|_{L^1_t(L^2)}
\le
C_{\eta}\|u\|_{L^\infty_t(L^2)}\|\nabla u\|_{L^2_t(L^2)}^2
+\eta\|u\|_{L^1_t(\dot B^2_{2,1})}
\end{split}
\eeno
for any positive constant $\eta$.
The same estimates holds for $T'_{\na\dot\D_ju}u.$
 Thus we obtain
\beq\label{a.4}
\begin{split}
\sum_{j\in\Z}
\|[\Dot\Delta_j\mathbb{P},u\cdot\nabla]u\|_{L^1_t(L^2)}
\le
C\|\nabla u\|_{L^2_T(L^2)}^2+C_\eta\|u\|_{L^\infty_t(L^2)}\|\nabla u\|_{L^2_t(L^2)}^2
+\eta\|u\|_{L^1_t(\dot B^2_{2,1})}.
\end{split}
 \eeq
Exactly along the same line to the proof of \eqref{a.4}, we  get, by
applying Bony's decomposition \eqref{bony}, that (for $a:=\frac{1}{\rho}-1$)
\beno
[\dot\Delta_j\mathbb{P}; {\rho}^{-1}]f=[\dot\Delta_j\mathbb{P},a]f
=[\dot\Delta_j\mathbb{P},T_a]f
+\dot\D_j\mathbb{P}T'_fa
-T'_{\dot\Delta_j\mathbb{P}f}a,
\eeno
It follows  again from Lemma 1 of \cite{Plan}  that
$$
\begin{aligned}
\sum_{j\in\Z}\|[\dot\Delta_j\mathbb{P},T_a]f\|_{L^2}
&\lesssim
\|\na a\|_{\dot B^{-1}_{\infty,1}}\|f\|_{L^2}
\lesssim
\|a\|_{\dot B^{0}_{\infty,1}}\|f\|_{L^2},
\end{aligned}
$$
and
$$
\sum_{j\in\Z}
\|\dot\D_j\mathbb{P}T_f'a\|_{L^2}
\lesssim
\sum_{\ell\geq j-3}2^{j\varepsilon}
\bigl\|\dot\D_\ell a\bigr\|_{L^{\frac{2}{\varepsilon}}}\|\dot
S_{\ell+2}f\|_{L^2}
\lesssim
\|a\|_{\dot B^{\varepsilon}_{\frac{2}{\varepsilon},1}}\|f\|_{L^2}.
$$
The same estimate holds for
$\bigl\|T'_{\dot\Delta_j\mathbb{P}f}a\bigr\|_{L^2}$.
Therefore, we obtain
\begin{equation}\label{a-a-5}
\sum_{j\in\Z}
\bigl\|[\dot\Delta_j\mathbb{P}; a]f\bigr\|_{L^1_t(L^2)}
\lesssim
\|a\|_{\widetilde L^\infty_t(\dot B^{\varepsilon}_{\frac{2}{\varepsilon},1})}\|f\|_{L^1_t(L^2)},
\end{equation}
from which, we deduce that
\begin{equation}\label{a.5}
\sum_{j\in\Z}
\|[\dot\Delta_j\mathbb{P};{\rho}^{-1}]\nabla\Pi\|_{L^1_tL^2}
\lesssim
\|a\|_{\widetilde L^\infty_t(\dot B^{\varepsilon}_{\frac{2}{\varepsilon},1})}
\|\nabla\Pi\|_{L^1_t(L^2)}.
\end{equation}
On the other hand, note that
$$
[\dot\Delta_j\mathbb{P}; {\rho}^{-1}]\Delta u=[\dot\Delta_j\mathbb{P},a]\Delta u
=[\dot\Delta_j\mathbb{P},a-\dot S_ka]\Delta u+[\dot\Delta_j\mathbb{P},\dot S_ka]\Delta u.
$$
Thanks to the inequality \eqref{a-a-5}, we deduce
\begin{equation}\label{estimate-1}
\sum_{j\in\Z}\|[\dot\Delta_j\mathbb{P},a-\dot S_ka]\Delta u \|_{L^1_tL^2}\lesssim
\|a-\dot S_ka\|_{\widetilde L^\infty_t(\dot B^{\varepsilon}_{\frac{2}{\varepsilon},1})}
\|\Delta u\|_{L^1_t(L^2)}.
\end{equation}
For $[\dot\Delta_j\mathbb{P},\dot S_ka]\Delta u$, Bony's decomposition implies
$$
[\dot\Delta_j\mathbb{P},\dot S_ka]\Delta u=[\dot\Delta_j\mathbb{P},T_{\dot S_k a}]\Delta u
+\dot\D_j\mathbb{P}T'_{\Delta u}\dot S_k a
-T'_{\dot\Delta_j\Delta u}\dot S_k a.
$$
It follows  again from Lemma 1 of \cite{Plan}  that
$$
\begin{aligned}
\sum_{j\in\Z}\|[\dot\Delta_j\mathbb{P},T_{\dot S_k a}]\Delta u \|_{L^2}
&\lesssim
\|\na\dot S_k  a\|_{L^\infty}\|\Delta u\|_{\dot B^{-1}_{2,1}}
\lesssim
2^k\|a\|_{L^\infty}\|u\|_{\dot B^{1}_{2,1}}
\\&
\le
\gamma\|\Delta u\|_{L^2}+C_\gamma2^{2k}\|a\|_{L^\infty}^2\|\nabla u\|_{L^2}
\end{aligned}
$$
for any positive constant $\gamma$,
and
$$
\begin{aligned}
\sum_{j\in\Z}
\|\dot\D_j\mathbb{P}T'_{\Delta u}\dot S_ka\|_{L^2}
&\lesssim
\sum_{\ell\geq j-3}2^{j\varepsilon}
\bigl\|\dot\D_\ell\dot S_k a\bigr\|_{L^{\frac{2}{\varepsilon}}}\|\dot
S_{\ell+2}\Delta u\|_{L^2}
\\&
\lesssim
\|\dot S_ka\|_{\dot B^{1+\varepsilon}_{\frac{2}{\varepsilon},1}}\|\nabla u\|_{L^2}
\lesssim
2^k\|a\|_{\dot B^{\varepsilon}_{\frac{2}{\varepsilon},1}}\|\nabla u\|_{L^2}.
\end{aligned}
$$
The same estimate holds for
$\bigl\|T'_{\dot\Delta_j\Delta u}\dot S_ma\bigr\|_{L^2}$.
Hence, we obtain
\begin{equation*}
\begin{aligned}
\sum_{j\in\Z}
\bigl\|[\dot\Delta_j\mathbb{P};\dot S_ka]\Delta u\bigr\|_{L^1_t(L^2)}
&\le
C\sqrt{t}2^k\|a\|_{\widetilde L^\infty_t(\dot B^{\varepsilon}_{\frac{2}{\varepsilon},1})}\|\na u\|_{L^2_t(L^2)}
+\gamma\|\Delta u\|_{L^1_t(L^2)}
\\&
+C_\gamma\sqrt{t}2^{2k}\|a\|_{L^\infty}^2\|\nabla u\|_{L^2_t(L^2)},
\end{aligned}
\end{equation*}
which along with \eqref{estimate-1} ensures that
\begin{equation}\label{a-b-5}
\begin{aligned}
\sum_{j\in\Z}
\bigl\|[\dot\Delta_j\mathbb{P}; a]\Delta u\bigr\|_{L^1_t(L^2)}
&\le
(\gamma+C\|a-\dot S_ka\|_{\widetilde L^\infty_t(\dot B^{\varepsilon}_{\frac{2}{\varepsilon},1})})\|\Delta u\|_{L^1_t(L^2)}
\\&
+C_\gamma\sqrt{t}2^{2k}\|a\|_{L^\infty}^2\|\nabla u\|_{L^2_t(L^2)}+C\sqrt{t}2^k\|a\|_{\widetilde L^\infty_t(\dot B^{\varepsilon}_{\frac{2}{\varepsilon},1})}\|\na u\|_{L^2_t(L^2)}.
\end{aligned}
\end{equation}
Substituting \eqref{a.4}, \eqref{a.5} and \eqref{a-b-5}
into \eqref{a.3}, and taking $\eta$ and $\gamma$ small enough, we obtain
\begin{equation}\label{a.7}
\begin{aligned}
\|u\|_{L^1_t(\dot B^{2}_{2,1})}
\leq
& C\sum_{j\in\Z}\bigl(1-e^{-ct2^{2j}}\bigr)\|\dot\Delta_ju_0\|_{L^2}
+C\|\nabla u\|_{L^2_t(L^2)}^2
\\&
+C\|a\|_{\widetilde L^\infty_t(\dot B^{\varepsilon}_{\frac{2}{\varepsilon},1})}
\|\nabla\Pi\|_{L^1_t(L^2)}+ C\|a-\dot S_ka\|_{\widetilde L^\infty_t(\dot B^{\varepsilon}_{\frac{2}{\varepsilon},1})}\|\Delta u\|_{L^1_t(L^2)}
\\&
+C\sqrt{t}2^k(1+\|a\|_{\widetilde L^\infty_t(\dot B^{\varepsilon}_{\frac{2}{\varepsilon},1})}).
\end{aligned}
\end{equation}

Thanks to \eqref{density-est-1} and \eqref{dens}, we get
\begin{equation}\label{density-est-1and2}
\begin{split}
&\|a\|_{\widetilde L^\infty_t(\dot B^{\varepsilon}_{\frac{2}{\varepsilon},1})}
+\|a\|_{\widetilde L^\infty_t(\dot B^{0}_{2,1})}
\lesssim
e^{C\|\nabla u\|_{L^1_t(L^\infty)}},\\
&\|a-\dot S_ka\|_{\widetilde{L}^\infty_t(\dot B^{\varepsilon}_{\frac{2}{\varepsilon},1})}\leq
\sum_{q\geq k}2^{q\varepsilon}\|\dot\Delta_q a_0\|_{L^{\frac{2}{\varepsilon}}}+\|a_0\|_{\dot B^{\varepsilon}_{\frac{2}{\varepsilon},1}} (e^{C\|\nabla u\|_{L^1_t(L^\infty)}}-1).
\end{split}
\end{equation}
Therefore, thanks to \eqref{Pi-L2-0} and \eqref{density-est-1and2}, we obtain from  \eqref{a.7} that
\begin{equation}\label{a.7}
\begin{aligned}
&\|u\|_{L^1_t(\dot B^{2}_{2,1})}
\leq
C\sum_{j\in\Z}\bigl(1-e^{-ct2^{2j}}\bigr)\|\dot\Delta_ju_0\|_{L^2}+Ce^{C\|\nabla u\|_{L^1_t(L^\infty)}}
\bigg(\sqrt{t}2^{k}
+\|\nabla u\|_{L^2_t(L^2)}^2\Bigr)
\\&\quad
+Ce^{C\|\nabla u\|_{L^1_t(L^\infty)}}\Bigl(\eta+M\sum_{q\geq k}\|\dot\Delta_q a_0\|_{L^\infty}+M\|a_0\|_{\dot
B^{0}_{\infty,1}}\{e^{C\|\nabla u\|_{L^1_t(L^\infty)}}-1\}\Bigr)\|u\|_{L^1_t(\dot B^2_{2,1})}\\
&\qquad+ C\bigg(\sum_{q\geq k}2^{q\varepsilon}\|\dot\Delta_q a_0\|_{L^{\frac{2}{\varepsilon}}}+\|a_0\|_{\dot B^{\varepsilon}_{\frac{2}{\varepsilon},1}} (e^{C\|\nabla u\|_{L^1_t(L^\infty)}}-1)\bigg)\|u\|_{L^1_t(\dot B^2_{2,1})}.
\end{aligned}
\end{equation}
By using \eqref{T-D} again, we deduce from the fact $\ell^1\hookrightarrow \ell^2$ that
$$
\begin{aligned}
\|\nabla u\|_{L^2_t(L^2)}
\lesssim
& \Bigl(\sum_{j\in\Z}\bigl(1-e^{-ct2^{2j}}\bigr)\|\dot\Delta_ju_0\|_{L^2}^2\Bigr)^{\f12}
+\Bigl(\sum_{j\in\Z}\|[\dot\Delta_j\mathbb{P};
u\cdot\nabla]u\|_{L^1_t(L^2)}^2\Bigr)^{\f12}
\\&
+\sum_{j\in\Z}
\|[\dot\Delta_j\mathbb{P};\frac{1}{\rho}]\bigl(\Delta u-\nabla\Pi\bigr)\|_{L^1_tL^2}.
\end{aligned}
$$
Similar to the proof of\eqref{a.4}, we may get
$$
\Bigl(\sum_{j\in\Z}\|[\dot\Delta_j\mathbb{P};
u\cdot\nabla]u\|_{L^1_t(L^2)}^2\Bigr)^{\f12}
\lesssim
\|\na u\|_{L^2_t(L^2)}^2.
$$
Hence, from \eqref{a.5}, \eqref{a-b-5}, and \eqref{Pi-L2-0}, it follows that
\begin{equation}\label{a.7b}
\begin{aligned}
&\|\nabla u\|_{L^2_t(L^2)}
\lesssim
\Bigl(\sum_{j\in\Z}\bigl(1-e^{-ct2^{2j}}\bigr)\|\dot\Delta_ju_0\|_{L^2}^2\Bigr)^{\f12}
+Ce^{C\|\nabla u\|_{L^1_t(L^\infty)}}
\bigg(\sqrt{t}2^{k}
+\|\nabla u\|_{L^2_t(L^2)}^2\Bigr)
\\&\quad
+Ce^{C\|\nabla u\|_{L^1_t(L^\infty)}}\Bigl(\eta+M\sum_{q\geq k}\|\dot\Delta_q a_0\|_{L^\infty}+M\|a_0\|_{\dot
B^{0}_{\infty,1}}\{e^{C\|\nabla u\|_{L^1_t(L^\infty)}}-1\}\Bigr)\|u\|_{L^1_t(\dot B^2_{2,1})}\\
&\qquad+ C\bigg(\sum_{q\geq k}2^{q\varepsilon}\|\dot\Delta_q a_0\|_{L^{\frac{2}{\varepsilon}}}+\|a_0\|_{\dot B^{\varepsilon}_{\frac{2}{\varepsilon},1}} (e^{C\|\nabla u\|_{L^1_t(L^\infty)}}-1)\bigg)\|u\|_{L^1_t(\dot B^2_{2,1})}.
\end{aligned}
\end{equation}
Combining \eqref{a.7} and \eqref{a.7b}, we achieve
\begin{equation}\label{a.7c}
\begin{aligned}
&\|u\|_{L^1_t(\dot B^{2}_{2,1})} +\|\nabla u\|_{L^2_t(L^2)}\\
&\leq
C\sum_{j\in\Z}\bigl(1-e^{-ct2^{2j}}\bigr)\|\dot\Delta_ju_0\|_{L^2}
+Ce^{C\|\nabla u\|_{L^1_t(L^\infty)}}
\bigg(\sqrt{t}2^{k}
+\|\nabla u\|_{L^2_t(L^2)}^2\Bigr)
\\&\quad
+Ce^{C\|\nabla u\|_{L^1_t(L^\infty)}}\Bigl(\eta+M\sum_{q\geq k}\|\dot\Delta_q a_0\|_{L^\infty}+M\|a_0\|_{\dot
B^{0}_{\infty,1}}\{e^{C\|\nabla u\|_{L^1_t(L^\infty)}}-1\}\Bigr)\|u\|_{L^1_t(\dot B^2_{2,1})}\\
&\qquad+ C\bigg(\sum_{q\geq k}2^{q\varepsilon}\|\dot\Delta_q a_0\|_{L^{\frac{2}{\varepsilon}}}+\|a_0\|_{\dot B^{\varepsilon}_{\frac{2}{\varepsilon},1}} (e^{C\|\nabla u\|_{L^1_t(L^\infty)}}-1)\bigg)\|u\|_{L^1_t(\dot B^2_{2,1})}.
\end{aligned}
\end{equation}
Consequently, taking $\eta$ small enough, $k$ large enough, and then $t$ sufficiently small in \eqref{a.7c}, we deduce
\eqref{a.1rf},
which completes the proof of of Proposition \ref{lema.1.1}.
\end{proof}
Based on this, we may get further estimate about the pressure.
\begin{prop}\label{prop-H-negtive}
\par Let $\varepsilon\in(0,1)$ and $a\in\dot B^{\varepsilon}_{\frac{2}{\varepsilon},1}(\R^2)$ such that $ 0<{\underline{b}}\leq
1+a\leq {\bar{b}}$, and
\begin{equation}\label{est-unique-1}
\|a-\dot S_k a\|_{\dot B^{\varepsilon}_{\frac{2}{\varepsilon},1}}\leq c
\end{equation}
for some sufficiently small positive constant $c$ and some integer
$k\in \Z.$  Let $F \in\dot B^{0}_{2,1}(\R^2)$  and $\grad \Pi \eqdefa
\mathcal{H}_{b}({F})\in\dot B^{0}_{2,1}(\R^2)$ solves
\begin{equation}\label{unique-per-model}
\dive\,((1+a) \grad \Pi) =\dive\,F.
\end{equation}
Then there holds
\begin{equation}\label{estimate-unique-per-model}
\|\grad \Pi\|_{\dot B^{0}_{2,1}}
\lesssim
\|F\|_{\dot B^{0}_{2,1}}
+\|a\|_{\dot B^{\varepsilon}_{\frac{2}{\varepsilon},1}}\|\nabla\Pi\|_{L^2}.
\end{equation}
\end{prop}
\begin{proof}
We first deduce from \eqref{est-unique-1} and ${\underline{b}}\leq
1+a$ that
\begin{equation}\label{lower-bdd-1}
1+\dot S_k a=1+a+(\dot S_k a-a) \geq{{\underline{b}}\over{2}}.
\end{equation}
We rewrite \eqref{unique-per-model} in the following form
\begin{equation*}
\dv[(1+\dot S_k a)\nabla\Pi]=\dv\,F+\dv\,[(\dot S_k a-a)\nabla\Pi],
\end{equation*}
and applying  $\dot\Delta_q$ to the above equation gives
\begin{equation*}
\dv\,[(1+\dot S_k a)\dot\Delta_q\nabla\Pi]=\dv\,\dot\Delta_q F+\dv\, \dot\Delta_q
[(\dot S_k a-a)\nabla\Pi]+\dv\,([\dot S_k a,\dot\Delta_q]\nabla\Pi).
\end{equation*}
Taking the $L^{2}$ inner product of this equation with $\dot\Delta_q\Pi,$
we obtain by a similar estimate of \eqref{a.5} that
\begin{equation*}\label{estimate-unique-per-model-1}
\begin{split}
 \|\grad \Pi\|_{\dot B^{0}_{2,1}}
 &\lesssim
 \|(\dot S_k a-a)\nabla\Pi\|_{\dot B^{0}_{2,1}}
+\|F\|_{\dot B^{0}_{2,1}}
+\sum_{q\in\Z}\bigl\|[\dot S_k a,\dot \Delta_q]\nabla\Pi\bigr\|_{L^2}
 \\
&\lesssim
\|\dot S_k a-a\|_{\dot B^{\varepsilon}_{\frac{2}{\varepsilon},1}}\|\nabla\Pi\|_{\dot B^{0}_{2,1}}
+\|F\|_{\dot B^{0}_{2,1}}
+\|a\|_{\dot B^{\varepsilon}_{\frac{2}{\varepsilon},1}}\|\nabla\Pi\|_{L^2},
\end{split}
\end{equation*}
which along with \eqref{est-unique-1} leads to \eqref{estimate-unique-per-model}.
 This completes the proof of the Proposition.
\end{proof}
Consequently, we may derive the main result of the section as follows.
\begin{prop}\label{prop-unif-bdd-1}
\par Under the assumptions of Proposition \ref{lema.1.1}, there holds that for any $t\in [0,T_1]$
\begin{equation}\label{momun-pre-0}
\begin{split}
\|u\|_{\widetilde L^\infty_t(\dot B^0_{2,1})}+
\|u\|_{L^1_t(\dot B^2_{2,1})}
+\|\partial_tu\|_{L^1_t(\dot B^{0}_{2,1})}
+\|\nabla\Pi\|_{L^1_t(\dot B^{0}_{2,1})}
\leq
C_0,
\end{split}
\end{equation}
where the constant $C_0$ depends only on the initial data $(\rho_0, \, u_0)$, and the positive time $T_1$ is determined by Proposition \ref{lema.1.1}.
\end{prop}

\begin{proof} Back to the proof of Proposition \ref{lema.1.1}, according to \eqref{AAA}, we have
\begin{equation}\label{momun-pre-1}
\dv\bigl\{(1+a)\nabla\Pi\bigr\}=\dv(a\Delta u)-\dv\bigl\{(u\cdot\nabla)u\bigr\}.
\end{equation}
Combining Proposition \ref{small-density-1} with Proposition \ref{lema.1.1}, we know that the inequality \eqref{est-unique-1} holds for any $t\in [0, T_1]$.  Then applying Proposition \ref{prop-H-negtive} to \eqref{momun-pre-1} yields that
\begin{equation*}
\|\grad \Pi\|_{\dot B^{0}_{2,1}}
\lesssim
\|a\Delta u\|_{\dot B^{0}_{2,1}}+\|(u\cdot\nabla)u\|_{\dot B^{0}_{2,1}}
+\|a\|_{\dot B^{\varepsilon}_{\frac{2}{\varepsilon},1}}\|\nabla\Pi\|_{L^2},
\end{equation*}
and then
\begin{equation}\label{momun-pre-2}
\|\grad \Pi\|_{L^1_t(\dot B^{0}_{2,1})}
\lesssim
\|a\Delta u\|_{L^1_t(\dot B^{0}_{2,1})}+\|(u\cdot\nabla)u\|_{L^1_t(\dot B^{0}_{2,1})}
+\|a\|_{L^\infty_t(\dot B^{\varepsilon}_{\frac{2}{\varepsilon},1})}\|\nabla\Pi\|_{L^1_t(L^2)}.
\end{equation}
Due to the product law in Besov spaces and the interpolation inequality, we get
\begin{equation*}
\begin{aligned}
&\|(u\cdot\nabla)u\|_{L^1_t(\dot B^{0}_{2,1})}
\lesssim\|\dv\,(u\,\otimes\,u)\|_{L^1_t(\dot B^{0}_{2,1})}
\lesssim\|u\,\otimes\,u\|_{L^1_t(\dot B^{1}_{2,1})}\lesssim
\int_0^t\|u\|_{\dot B^{1}_{2,1}}^2\,d\tau\\
&
\lesssim
\int_0^t\|u\|_{L^2}\|\Delta u\|_{L^2}\,d\tau\lesssim
\|u\|_{L^\infty_t(L^2)}\|\Delta u\|_{L^1_t(L^2)} \lesssim
\|u\|_{L^\infty_t(L^2)}\|u\|_{L^1_t(\dot B^2_{2,1})},\\
\end{aligned}
\end{equation*}
\begin{equation*}
\begin{aligned}
&\|a\Delta u\|_{L^1_t(\dot B^{0}_{2,1})}
\lesssim
\|a\|_{\widetilde L^\infty_t(\dot B^{\varepsilon}_{\frac{2}{\varepsilon},1})}\|u\|_{L^1_t(\dot B^2_{2,1})}
\end{aligned}
\end{equation*}
Hence, thanks to Propositions \ref{Prop1}, \ref{lema.1.1}, and \eqref{density-est-1}, we achieve
\begin{equation}\label{momun-pre-3}
\begin{aligned}
&\|a\|_{\widetilde L^\infty_t(\dot B^{\varepsilon}_{\frac{2}{\varepsilon},1})}
\leq C_0,\quad
\|\nabla\Pi\|_{L^1_t(L^2)}\leq C_0, \quad \|u\|_{L^1_t(\dot B^2_{2,1})}\leq
C_0, \\
&
\|a\Delta u\|_{L^1_t(\dot B^{0}_{2,1})}
\leq
C_0, \quad\|(u\cdot\nabla)u\|_{L^1_t(\dot B^{0}_{2,1})} \leq C_0.
\end{aligned}
\end{equation}
Inserting \eqref{momun-pre-3} into \eqref{momun-pre-2} ensures that
\begin{equation}\label{momun-pre-4}
\|\nabla\Pi\|_{L^1_t(\dot B^{0}_{2,1})}
\leq
C_0.
\end{equation}
On the other hand, thanks to \eqref{T-D}, \eqref{a.4}, \eqref{a.5} and \eqref{a-b-5}, we readily deduce that
\begin{equation}\label{momun-pre-5}
\|u\|_{\widetilde L^\infty_t(\dot B^0_{2,1})}
\leq
C_0.
\end{equation}
While from the momentum equations in \eqref{1.1} and \eqref{momun-pre-3}, one has
\begin{equation}
\begin{split}
 \|\pa_t u\|_{L^1_t(\dot B^{0}_{2,1})}&\lesssim \|(u\cdot\nabla)u \|_{L^1_t(\dot B^{0}_{2,1})}+\|\Delta u-\grad\Pi\|_{L^1_t(\dot B^{0}_{2,1})}+\|a(\Delta u-\grad\Pi)\|_{L^1_t(\dot B^{0}_{2,1})}\\
 &\leq\,C_0+C(1+\|a\|_{\widetilde L^\infty_t(\dot B^{\varepsilon}_{\frac{2}{\varepsilon},1})})(\|u\|_{L^1_t(\dot B^2_{2,1})}+\|\grad\Pi\|_{L^1_t(\dot B^{0}_{2,1})})\leq C_0,
\end{split}
\end{equation}
which follows \eqref{momun-pre-0}. This ends the proof of Proposition \ref{prop-unif-bdd-1}.
\end{proof}

\renewcommand{\theequation}{\thesection.\arabic{equation}}
\setcounter{equation}{0}
\section{The proof of Theorem \ref{thm1.1}}

We now turn to the proof of Theorem \ref{thm1.1}.

\begin{proof}[Proof  of Theorem \ref{thm1.1}] We divide the proof
into two steps.\smallskip

\no{\bf Step 1.} Existence of strong solutions.

\smallskip

 Given $\rho_0$ with $a_0:=\frac{1}{\rho_0}-1\in\dot B^{\varepsilon}_{\frac{2}{\varepsilon},1}(\R^2)$ and satisfying
\eqref{thma.1},  $u_0\in\dot B^0_{2,1}(\R^2),$ we first mollify the initial data
to be
\beq \label{o.1} a_{0,n}\eqdefa\,a_0\ast j_n,\andf
u_{0,n}\eqdefa u_0\ast j_n,\eeq where
$j_n(|x|)=n^{2}j(|x|/n)$ is the standard Friedrich's mollifier. Then
we deduce  from the standard well-posedness theory of inhomogeneous
Navier-Stokes system (see \cite{danchin04} for instance) that
\eqref{1.1}  has a unique
solution $(\rho_n,u_n,\na\Pi_n)$ on $[0,T^\ast_n[$ for some positive
time $T^\ast_n.$ It is easy to observe from \eqref{o.1} that
\beno
\begin{split}
\|a_{0,n}\|_{\dot B^{\varepsilon}_{\frac{2}{\varepsilon},1}}
\le
C\|a_0\|_{\dot B^{\varepsilon}_{\frac{2}{\varepsilon},1}}
\andf
\|u_{0,n}\|_{\dot B^{0}_{2,1}}\leq C\|u_0\|_{B^{0}_{2,1}},
\end{split}
\eeno
so, under
the assumptions of Theorem \ref{thm1.1}, we infer from Lemma
\ref{lema.1.1} that there holds \beq\label{o.2}
\begin{aligned}
\|u_n\|_{L^1_t(\dot B^2_{2,1})}\leq C_0\andf
\|a_{n}\|_{\widetilde L^\infty_t(\dot B^{\varepsilon}_{\frac{2}{\varepsilon},1})}
\leq C_0,
\end{aligned}
\eeq for $t< T^\ast_n.$
 Without loss of generality, we may assume $T^\ast_n$ is the lifespan of the
approximate solutions $(\rho_n,u_n,\na\Pi_n).$ Then by virtue of
\cite{danchin04} and \eqref{o.2}, we conclude that
$T_n^\ast \geq \,T_1$ for some positive constant $T_1$. With \eqref{o.2}, we get, by using a standard
compactness argument, that \eqref{1.1} has a solution
$(\rho,u,\na\Pi)$ so that
$a\in C([0,T_1[;\,\dot B^{\varepsilon}_{\frac{2}{\varepsilon},1}(\R^2)),$
$u\in C([0,T_1[;\,\dot B^{0}_{2,1}(\R^2))\cap  L^1([0,T_1[;\,
\dot B^{2}_{2,1}(\R^2)),$ $\p_tu, \na\Pi\in L^1_{loc}([0,T_1[;\dot B^0_{2,1}(\R^2)).$
Furthermore, we can find some  $t_0\in
(0,T_1)$ such that $u(t_0)\in H^1(\R^2)$. Based on the initial data $a(t_0)\in \dot B^{\varepsilon}_{\frac{2}{\varepsilon},1}(\R^2)$ and $u(t_0)\in H^1(\R^2)$, we may deduce the global existence of the solution to \eqref{1.1}
according to \cite{DM, PZZ}. This completes the
proof of the existence of the global solution to \eqref{1.1}.

\smallskip

\no{\bf Step 2.} Uniqueness of strong solutions.\

Let's first say from \eqref{density-est-1} that for any $t>0$
\begin{equation}\label{T-1}
\|a\|_{\widetilde L^\infty_t(\dot B^1_{2,1})}\le\|a_0\|_{\dot B^1_{2,1}}e^{C\|u\|_{L^1_t(\dot B^2_{2,1})}}.
\end{equation}

Let $(\rho^i,u^i,\nabla\Pi^i)$ with $i=1,2$ be two solutions of
\eqref{1.1} which satisfies \eqref{AG} and $\rho=\frac{1}{1+a}.$
We denote \beno (\delta a,\delta
u,\nabla\delta\Pi) \eqdefa
(a^2-a^1,u^2-u^1,\nabla\Pi^2-\nabla\Pi^1).
\eeno
Then the system
for $(\delta a,\delta u,\nabla\delta\Pi)$ reads
\beq\label{u.3}
\left\{\begin{array}{l}
\displaystyle \pa_t\delta a+u^2\cdot\nabla\delta a=-\delta u\cdot\nabla a^1\\
\displaystyle \pa_t\delta u+(u^2\cdot\nabla)\delta u
-(1+a^2)(\Delta\delta u-\grad\delta\Pi)=\d F, \\
\displaystyle \dv\,\delta u = 0, \\
\displaystyle (\delta a,\delta u)|_{t=0}=(0,0),
\end{array}
\right. \eeq
where $\d F$ is determined by
$$
\begin{aligned}
\d F=-(\delta u\cdot\nabla)u^1+\delta a(\Delta u^1-\nabla\Pi^1).
\end{aligned}
$$
For $\delta u,$ we first write the momentum equation of \eqref{u.3} as
\begin{equation}\label{Pre-1}
\partial_t\delta u+(u^2\cdot\nabla)\delta
u-(1+S_k a^2)(\Delta\delta u-\nabla\delta \Pi) =H
\end{equation}
with
$$
H=(a^2-S_k a^2)(\Delta\delta u-\nabla\delta \Pi) -\delta
u\cdot\nabla u^1+\delta a(\Delta u^1-\nabla\Pi^1).
$$
Applying Proposition
\ref{prop-uniqueness-1} to \eqref{Pre-1} yields that for
$\forall \, 0<t\leq T$
\begin{equation}\label{estimate-uniqueness-velosity-1}
\begin{aligned}
\|\delta u\|_{{L}^{\infty}_t(B^{-1}_{2,\infty})}&+\|\delta
u\|_{\widetilde{L}^1_t(B^{1}_{2, \infty})}
\le
Ce^{CT2^k}
\bigl(\|\nabla\delta\Pi\|_{\widetilde L^1_T(B^{-1}_{2,\infty})}
+\|H\|_{\widetilde L^1_T(B^{-1}_{2,\infty})}\bigr).
\end{aligned}
\end{equation}
On the other hand, applying  $\dv$ to the momentum equation of \eqref{u.3} yields
\begin{equation}\label{est-uniq-2}
\dv[(1+a^2)\nabla\delta\Pi]=\dv\,G
\end{equation}
with
$$
\begin{aligned}
G=&a^2\Delta\delta u-\delta u\cdot\nabla
u^1-u^2\cdot\nabla\delta u +\delta a(\Delta u^1-\nabla\Pi^1)
\\
=&(a^2-S_ma^2)\Delta\delta u+S_ma^2\Delta\delta u-\delta u\cdot\nabla
u^1-u^2\cdot\nabla\delta u +\delta a(\Delta u^1-\nabla\Pi^1)
\\
\eqdefa &\sum_{\ell=0}^5\mbox{I}_{\ell}.
\end{aligned}
$$
Thanks to Propositions \ref{lema.1.1} and \ref{small-density-1},  we get that, for any small constant $c_0>0$, there exist sufficiently large $j_0 \in \mathbb{N}$ and a positive existence time $T_1$ such that $ \|a^2-S_ja^2\|_{\widetilde L^\infty_t(B^1_{2,1})} <c_0$ for any $j\geq\,j_0$ and $t\in [0, T_1]$.
Then applying Proposition \ref{prop-H-negtive} to \eqref{est-uniq-2} leads to
\begin{equation*}\label{uniqueness-pressure-elli-1-a}
\begin{split}
\|\grad\delta\Pi\|_{\widetilde{L}^{1}_{t} ({B}^{-1}_{2, \infty})}
\lesssim
 \bigl(1+2^{j}\|a^2\|_{\widetilde L^\infty_t(B^{1}_{2,1})}^2\bigr)
 \bigl(\|G\|_{\widetilde L^1_t(B^{-2}_{2,\infty})}
 +\|\dv\,G\|_{\widetilde L^1_t(B^{-2}_{2,\infty})}\bigr).
\end{split}
\end{equation*}
While by Lemma \ref{lem2.1} and Bony's decomposition, one can see
$$
\|\mbox{I}_1\|_{\widetilde L^1_t(B^{-2}_{2,\infty})}
+\|\dv\,\mbox{I}_1\|_{\widetilde L^1_t(B^{-2}_{2,\infty})}
\lesssim
\|\mbox{I}_1\|_{\widetilde L^1_t(B^{-1}_{2,\infty})}
\lesssim
\|a^2-S_ma^2\|_{\widetilde L^\infty_t(B^1_{2,1})}
\|\delta u\|_{\widetilde L^1_t(B^1_{2,\infty})},
$$
$$
\begin{aligned}
\|\mbox{I}_2\|_{\widetilde L^1_t(B^{-2}_{2,\infty})}
&+\|\dv\,\mbox{I}_2\|_{\widetilde L^1_t(B^{-2}_{2,\infty})}
\lesssim
\|T_{S_ma^2}\Delta\delta u\|_{\widetilde L^1_t(B^{-2}_{2,\infty})}
+\|T_{\Delta\delta u}S_ma^2\|_{\widetilde L^1_t(B^{-2}_{2,\infty})}
\\&
+\|\cR(S_ma^2,\Delta\delta u)\|_{\widetilde L^1_t(B^{-1}_{2,\infty})}
+\|T_{\nabla S_ma^2}\Delta\delta u\|_{\widetilde L^1_t(B^{-2}_{2,\infty})}
+\|T_{\Delta\delta u}\nabla S_ma^2\|_{\widetilde L^1_t(B^{-2}_{2,\infty})}
\\&
\lesssim
2^m
\|\delta u\|_{\widetilde L^1_t(B^0_{2,\infty})}.
\end{aligned}
$$
Similarly, one has
$$
\begin{aligned}
&\|(\mbox{I}_3,\mbox{I}_4)\|_{\widetilde L^1_t(B^{-2}_{2,\infty})}
+\|\dv\,(\mbox{I}_3,\mbox{I}_4)\|_{\widetilde L^1_t(B^{-2}_{2,\infty})}\\
&\lesssim
\|\mbox{I}_3\|_{\widetilde L^1_t(B^{-1}_{2,\infty})}
+\|T_{u^2}\nabla\delta u\|_{\widetilde L^1_t(B^{-2}_{2,\infty})}
+\|T_{\nabla\delta u}u^2\|_{\widetilde L^1_t(B^{-2}_{2,\infty})}+\|\cR(u^2_i,\partial_i\delta u)\|_{\widetilde L^1_t(B^{-1}_{2,\infty})}
\\&\qquad+\|T_{\nabla u^2}\nabla\delta u\|_{\widetilde L^1_t(B^{-2}_{2,\infty})}
+\|T_{\nabla\delta u}\nabla u^2\|_{\widetilde L^1_t(B^{-2}_{2,\infty})}
+\|\cR(\partial_{\ell}u^2_i,\partial_i\delta u_{\ell})\|_{\widetilde L^1_t(B^{-2}_{2,\infty})}
\\&
\lesssim
\int_0^t\|\delta u\|_{B^{-1}_{2,\infty}}
\bigl(\|u^1\|_{B^1_{\infty,1}}+\|u^2\|_{B^1_{\infty,1}}\bigr)d\tau
\end{aligned}
$$
and
$$
\begin{aligned}
\|\mbox{I}_5\|_{\widetilde L^1_t(B^{-2}_{2,\infty})}
+\|\dv\,\mbox{I}_5\|_{\widetilde L^1_t(B^{-2}_{2,\infty})}
\lesssim
\int_0^t\|\delta a\|_{L^2}
\bigl(\|\Delta u^1\|_{L^2}+\|\nabla\Pi^1\|_{L^2}\bigr)d\tau.
\end{aligned}
$$
Thus, we obtain
\begin{equation}\label{uniqueness-pressure-elli-1-a}
\begin{split}
\|\grad\delta\Pi\|_{\widetilde{L}^{1}_{t} ({B}^{-1}_{2, \infty})}
\lesssim
 \Bigl\{1+&2^{j}\|a^2\|_{\widetilde L^\infty_t(B^{1}_{2,1})}^2\Bigr\}
 \Bigl\{
 \|a^2-S_ma^2\|_{\widetilde L^\infty_t(B^1_{2,1})}
\|\delta u\|_{\widetilde L^1_t(B^1_{2,\infty})}
\\&
+2^m\|\delta u\|_{\widetilde L^1_t(B^0_{2,\infty})}
 +\int_0^t\|\delta u\|_{B^{-1}_{2,\infty}}
\bigl(\|u^1\|_{B^1_{\infty,1}}+\|u^2\|_{B^1_{\infty,1}}\bigr)d\tau
\\&\qquad
 +\int_0^t\|\delta a\|_{L^2}
\bigl(\|\Delta u^1\|_{L^2}+\|\nabla\Pi^1\|_{L^2}\bigr)d\tau
  \Bigr\}.
\end{split}
\end{equation}
Toward the estimate of $\|H\|_{\widetilde L^1_t(B^{-1}_{2,\infty})}$, by using Bony's decomposition again, we get
$$
\begin{aligned}
\|H\|_{\widetilde L^1_t(B^{-1}_{2,\infty})}
&\lesssim
\|a^2-S_k a^2\|_{\widetilde L^\infty_t(B^1_{2,1})}
\bigl(\|\Delta\delta u\|_{\widetilde L^1_t(B^{-1}_{2,\infty})}+
\|\nabla\delta \Pi\|_{\widetilde L^1_t(B^{-1}_{2,\infty})}\bigr)
\\&
+\int_0^t\|\delta u\|_{B^{-1}_{2,\infty}}\|u^1\|_{B^1_{\infty,1}}d\tau
+\int_0^t\|\delta a\|_{L^2}\bigl(\|\Delta u^1\|_{L^2}+\|\nabla\Pi^1\|_{L^2}\bigr)d\tau.
\end{aligned}
$$
Taking $k_0$ sufficiently large and $0<T_2(\leq T_1)$ small enough, one may achieve, due to  \eqref{T-1}, that, for any $k\geq\,k_0$ and $t\in (0, T_2]$
\begin{equation}\label{small-refe-1}
\|a^2- S_ka^2\|_{\widetilde L^{\infty}_T({B}^{1}_{2,1})} \le c_0,
\end{equation}
Therefore, thanks to \eqref{estimate-uniqueness-velosity-1}, \eqref{uniqueness-pressure-elli-1-a}
and \eqref{small-refe-1}, we prove
$$
\begin{aligned}
\|\delta u\|_{{L}^{\infty}_t(B^{-1}_{2,\infty})}&+\|\delta
u\|_{\widetilde{L}^1_t(B^{1}_{2, \infty})}
\lesssim
\Bigl\{1+2^{j}\|a^2\|_{\widetilde L^\infty_t(B^{1}_{2,1})}^2\Bigr\}
 \Bigl\{
 \|a^2-S_ma^2\|_{\widetilde L^\infty_t(B^1_{2,1})}
\|\delta u\|_{\widetilde L^1_t(B^1_{2,\infty})}
\\&
+2^m\|\delta u\|_{\widetilde L^1_t(B^0_{2,\infty})}
 +\int_0^t\|\delta u\|_{B^{-1}_{2,\infty}}
\bigl(\|u^1\|_{B^1_{\infty,1}}+\|u^2\|_{B^1_{\infty,1}}\bigr)d\tau
\\&
 +\int_0^t\|\delta a\|_{L^2}
\bigl(\|\Delta u^1\|_{L^2}+\|\nabla\Pi^1\|_{L^2}\bigr)d\tau
  \Bigr\}.
\end{aligned}
$$
Taking $m_0$ sufficiently large and the positive time $T_3 (\leq\, T_2)$ small enough, we obtain that, for any $m \geq\,m_0$ and $t\in (0, T_3]$
\begin{equation}\label{diff-velocity-est-1}
\begin{aligned}
&\|\delta u\|_{{L}^{\infty}_t(B^{-1}_{2,\infty})}+\|\delta
u\|_{\widetilde{L}^1_t(B^{1}_{2, \infty})}
\lesssim
\Bigl\{1+2^{j}\|a^2\|_{\widetilde L^\infty_t(B^{1}_{2,1})}^2\Bigr\}
 \Bigl\{
2^m\|\delta u\|_{\widetilde L^1_t(B^0_{2,\infty})}
 \\&\qquad\quad
 +\int_0^t\|\delta u\|_{B^{-1}_{2,\infty}}
\bigl(\|u^1\|_{B^1_{\infty,1}}+\|u^2\|_{B^1_{\infty,1}}\bigr)d\tau
 +\int_0^t\|\delta a\|_{L^2}
\bigl(\|\Delta u^1\|_{L^2}+\|\nabla\Pi^1\|_{L^2}\bigr)d\tau
  \Bigr\}.
\end{aligned}
\end{equation}
On the other hand, by a classical estimate of the transport equation, we get from the first equation in \eqref{u.3} that
\begin{equation}\label{diff-density-esti-1}
\|\delta a(\tau)\|_{L^2}
\le
\int_0^{\tau}\|(\delta u\cdot\nabla)a_1\|_{ L^2}ds
\lesssim
\int_0^{\tau}\|\delta u\|_{L^\infty}ds.
\end{equation}
While thanks to the interpolation inequality, one may prove that
$$
\begin{aligned}
2^m\bigl(1+2^{j}\|a^2\|_{\widetilde L^\infty_t(B^{1}_{2,1})}^2\bigr)
& \|\delta u\|_{\widetilde L^1_t(B^0_{2,\infty})}
\lesssim
2^m\bigl(1+2^{j}\|a^2\|_{\widetilde L^\infty_t(B^{1}_{2,1})}^2\bigr)
\|\delta u\|_{\widetilde L^1_t(B^{-1}_{2,\infty})}^{\f12}
\|\delta u\|_{\widetilde L^1_t(B^{1}_{2,\infty})}^{\f12}
\\&
\le
\eta\|\delta u\|_{\widetilde L^1_t(B^{1}_{2,\infty})}
+c_{\eta}2^{2m}\bigl(1+2^{2j}\|a^2\|_{\widetilde L^\infty_t(B^{1}_{2,1})}^4\bigr)
\int_0^t\|\delta u\|_{B^{-1}_{2,\infty}}d\tau
\end{aligned}
$$
As a result, we get
\begin{equation*}
\begin{aligned}
&\|\delta u\|_{{L}^{\infty}_t(B^{-1}_{2,\infty})}+\|\delta
u\|_{\widetilde{L}^1_t(B^{1}_{2, \infty})}\\
&\le
C(a_0,k,j,m,\eta)
\bigl\{\int_0^t\|\delta u\|_{B^{-1}_{2,\infty}}
\bigl(1+\|u^1\|_{B^1_{\infty,1}}+\|u^2\|_{B^1_{\infty,1}}\bigr)d\tau
\\&\qquad\qquad\qquad\qquad\qquad
 +\int_0^t\|\delta u\|_{L^1_{\tau}(L^\infty)}
\bigl(\|\Delta u^1\|_{L^2}+\|\nabla\Pi^1\|_{L^2}\bigr)d\tau
  \bigr\},
\end{aligned}
\end{equation*}
then for $t$ small enough, we obtain
\begin{equation}\label{lip-est-0}
\|\delta u\|_{{L}^{\infty}_t(B^{-1}_{2,\infty})}+\|\delta
u\|_{\widetilde{L}^1_t(B^{1}_{2, \infty})}
\lesssim
\int_0^t\|\delta u\|_{L^1_{\tau}(L^\infty)}
\bigl(\|\Delta u^1\|_{L^2}+\|\nabla\Pi^1\|_{L^2}\bigr)d\tau.
\end{equation}
Let $N$ be an arbitrary positive integer which will be determined later on, then
$$
\begin{aligned}
\|\delta u\|_{L^1_{\tau}(L^\infty)}
&\le
\|\delta u\|_{L^1_{\tau}(\dot B^0_{\infty,1})}\\
&\le
\sum_{q\le-N}\|\dot\Delta\delta u\|_{L^1_{\tau}(L^\infty)}
+\sum_{1-N\le q\le N}\|\dot\Delta\delta u\|_{L^1_{\tau}(L^\infty)}
+\sum_{q\geq N+1}\|\dot\Delta\delta u\|_{L^1_{\tau}(L^\infty)}.
\end{aligned}
$$
Hence, due to Bernstein's inequality, we infer
$$
\begin{aligned}
\|\delta u\|_{L^1_{\tau}(L^\infty)}
&\lesssim
2^{-N}\|\delta u\|_{L^1_{\tau}(L^2)}
+N\|\delta u\|_{\widetilde L^1_{\tau}(\dot B^1_{2,\infty})}
+2^{-N}\|\nabla\delta u\|_{L^1_{\tau}(L^\infty)}
\\&
\lesssim
2^{-N}\|\delta u\|_{L^1_{\tau}(L^2)}
+N\|\delta u\|_{\widetilde L^1_{\tau}(B^1_{2,\infty})}
+2^{-N}\|\nabla\delta u\|_{L^1_{\tau}(L^\infty)}
\end{aligned}
$$
If we choose $N$ such that
$$
N\approx \ln\bigl(e+\frac{\|\delta u\|_{L^1_{\tau}(L^2)}+\|\nabla\delta u\|_{L^1_{\tau}(L^\infty)}}
{\|\delta u\|_{\widetilde L^1_{\tau}(B^1_{2,\infty})}}\bigr),
$$
then there holds
$$
\begin{aligned}
\|\delta u\|_{L^1_{\tau}(L^\infty)}
&\lesssim
\|\delta u\|_{\widetilde L^1_{\tau}(B^1_{2,\infty})}
\ln\bigl(e+\frac{\|\delta u\|_{L^1_{\tau}(L^2)}+\|\nabla\delta u\|_{L^1_{\tau}(L^\infty)}}
{\|\delta u\|_{\widetilde L^1_{\tau}(B^1_{2,\infty})}}\bigr),
\end{aligned}
$$
and then
\begin{equation}\label{infty-est-1}
\begin{aligned}
\|\delta u\|_{L^1_{\tau}(L^\infty)}
\lesssim
\|\delta u\|_{\widetilde L^1_{\tau}(B^1_{2,\infty})}
\ln\bigl(e+\frac{\sum_{i=1}^2\{\tau\|u^i\|_{L^\infty_{\tau}(L^2)}
+\|\nabla u^i\|_{L^1_{\tau}(L^\infty)}\}}
{\|\delta u\|_{\widetilde L^1_{\tau}(B^1_{2,\infty})}}\bigr).
\end{aligned}
\end{equation}
Notice that for $\alpha\geq 0$ and $x\in(0,1],$ there holds
$$
\ln(e+\alpha x^{-1})\le\ln(e+\alpha)(1-\ln x).
$$
Thus, plugging \eqref{infty-est-1} into \eqref{lip-est-0} leads to
\begin{equation}\label{infty-est-2}
\begin{aligned}
&\|\delta u\|_{{L}^{\infty}_t(B^{-1}_{2,\infty})}
+\|\delta u\|_{\widetilde{L}^1_t(B^{1}_{2, \infty})}\\
&\lesssim
\int_0^t\|\delta u\|_{\widetilde{L}^1_{\tau}(B^{1}_{2, \infty})}
\bigl(1-\ln\|\delta u\|_{\widetilde{L}^1_{\tau}(B^{1}_{2, \infty})}\bigr)
\bigl(\|\Delta u^1\|_{L^2}+\|\nabla\Pi^1\|_{L^2}\bigr)d\tau.
\end{aligned}
\end{equation}
As $\int_0^1\frac{dx}{x(1-\ln x)}=+\infty,$ and $\|\Delta u^1\|_{L^2}+\|\nabla\Pi^1\|_{L^2}$ is locally integral in $\mathbb{R}^+$,
then by Osgood's lemma (Lemma \ref{lem-Osgood}), we obtain that $\delta u(t)=0$, which together with
\eqref{diff-density-esti-1} and \eqref{uniqueness-pressure-elli-1-a} implies that
$\delta a(t)=\delta\nabla \Pi(t)=0$ for all $t\in[0,T]$ with $T$ small. Applying an inductive argument implies that $\delta u(t)=\delta a(t)=\delta\nabla \Pi(t)=0$ for all $t >0$.

Furthermore, applying \eqref{infty-est-2} (up to a slight modification) to the system \eqref{1.1general}, we may readily prove that the solution $(a,\, u) \in C(\R_+;\,B^1_{2, 1}(\R^2))\times C(\R_+;\,\dot B^{0}_{2,1}(\R^2))$ depends continuously on the initial data $(a_0,\, u_0) \in B^1_{2, 1}(\R^2)\times \dot B^{0}_{2,1}(\R^2)$. This completes the proof of Theorem \ref{thm1.1}.
\end{proof}

\appendix

\setcounter{equation}{0}
\section{Littlewood-Paley analysis}\label{abbendixb}

The proof  of Theorem \ref{thm1.1}
requires a dyadic decomposition of the Fourier variables, which is
called the Littlewood-Paley decomposition. Let us briefly explain how
it may be built in the case $x\in\R^2$ (see e.g. \cite{BCD}). Let
$\varphi$ be a smooth function  supported in the ring
$\mathcal{C}\eqdefa \{
\xi\in\R^2,\frac{3}{4}\leq|\xi|\leq\frac{8}{3}\}$  and $\chi(\xi)$
be a smooth function  supported in the ball $\mathcal{B}\eqdefa \{
\xi\in\R^2,\ |\xi|\leq\frac{4}{3}\}$ such that
\begin{equation*}
 \sum_{j\in\Z}\varphi(2^{-j}\xi)=1 \quad\hbox{for}\quad \xi\neq
 0\quad\mbox{and}\quad \chi(\xi)+ \sum_{q\geq 0}\varphi(2^{-q}\xi)=1\quad\hbox{for \ all }\quad
 \xi\in\R^2.
\end{equation*}
Now, for $u\in{\mathcal S}'(\R^2),$ we set
\begin{equation}\label{LP-decom-sum-1}
  \begin{aligned}
&\forall q\in\Z,\quad \dot\Delta_qu=\varphi(2^{-q}\textnormal{D})u\hspace{1cm}\mbox{and}
\hspace{1cm}
\dot S_qu=\sum_{j\leq q-1}\Delta_{j}u.
\\&
q\geq0,\quad\Delta_qu=\varphi(2^{-q}\textnormal{D})u,\quad
\Delta_{-1}u=\chi(\textnormal{D})u\quad\mbox{and}\quad
S_qu=\sum_{-1\le q'\le q-1}\Delta_{q'}u.
\end{aligned}
\end{equation}
We have the formal decomposition
$$
u=\sum_{q\in\Z}\dot\Delta_q \,u,\quad\forall\,u\in {\mathcal {S}}'(\R^2)/{\mathcal{P}}[\R^2]
\quad\mbox{and}\quad
u=\sum_{q\geq-1}\Delta_q \,u,\quad\forall\,u\in {\mathcal {S}}'(\R^2),
$$
where ${\mathcal{P}}[\R^2]$ is the set of polynomials (see \cite{PE}).
Moreover, the Littlewood-Paley decomposition satisfies the
property of almost orthogonality:
\begin{equation}\label{Pres_orth}
\begin{aligned}
&\dot\Delta_k\dot\Delta_q u\equiv 0
\quad\mbox{if}\quad\vert k-q\vert\geq 2
\quad\mbox{and}\quad\dot\Delta_k(\dot S_{q-1}u\dot\Delta_q u)
\equiv 0\quad\mbox{if}\quad\vert k-q\vert\geq 5,
\\&
\Delta_k\Delta_q u\equiv 0
\quad\mbox{if}\quad\vert k-q\vert\geq 2
\quad\mbox{and}\quad\Delta_k( S_{q-1}u\Delta_q u)
\equiv 0\quad\mbox{if}\quad\vert k-q\vert\geq 5.
\end{aligned}
\end{equation}
We recall now the definition of nonhomogeneous and homogeneous Besov spaces from
\cite{Tri}.
\begin{defi}\label{def1.1}
{\sl  Let $(p,r)\in[1,+\infty]^2,$ $s\in\R$ and $u\in{\mathcal
S}'(\R^2),$ we set
$$
\|u\|_{B^s_{p,r}}\eqdefa\Big(2^{qs}\|\Delta_q
u\|_{L^{p}}\Big)_{\ell ^{r}}
\quad\mbox{and}\quad
\|u\|_{\dot B^s_{p,r}}\eqdefa\Big(2^{qs}\|\dot\Delta_q
u\|_{L^{p}}\Big)_{\ell ^{r}},
$$
with the usual modification if $r=\infty.$
\begin{itemize}
\item
For $s\in\R,$ we define
$B^s_{p,r}(\R^2)\eqdefa \big\{u\in{\mathcal S}'(\R^2)\;\big|\; \Vert
u\Vert_{B^s_{p,r}}<\infty\big\}.$

\item
For $s<\frac{2}{p}$ (or $s=\frac{2}{p}$ if $r=1$), we  define $\dot
B^s_{p,r}(\R^2)\eqdefa \big\{u\in{\mathcal S}'(\R^2)\;\big|\; \Vert
u\Vert_{\dot B^s_{p,r}}<\infty\big\}.$

\item
If $k\in\N$ and $\frac{2}{p}+k\leq s<\frac{2}{p}+k+1$ (or
$s=\frac{2}{p}+k+1$ if $r=1$), then $\dot B^s_{p,r}(\R^2)$ is
defined as the subset of distributions $u\in{\mathcal S}'(\R^2)$
such that $\partial^\beta u\in\dot B^{s-k}_{p,r}(\R^2)$ whenever
$|\beta|=k.$
\end{itemize}}
\end{defi}

\begin{rmk}\label{rmk1.1}
\begin{enumerate}
  \item We point out that if $s>0$ then $B^s_{p,r}=\dot B^s_{p,r}\cap L^p$ and
$$
\|u\|_{B^s_{p,r}}\approx \|u\|_{\dot B^s_{p,r}}+\|u\|_{L^p}.
$$
\item It is easy to verify that the homogeneous Besov
space $\dot{B}^s_{2,2}(\R^2)$ (resp. ${B}^s_{2,2}(\R^2)$) coincides with the classical
homogeneous Sobolev space $\dot{H}^{s}(\R^2)$ (resp. $H^s(\R^2)$) and
$\dot{B}^s_{\infty,\infty}(\R^2)$ coincides with the classical
homogeneous H\"older space $\dot{C}^s(\R^2)$  when $s$ is not
positive integer, in case $s$ is a nonnegative integer,
$\dot{B}^s_{\infty,\infty}(\R^2)$ coincides with the classical
homogeneous Zygmund space $\dot{C}^s_{\ast}(\R^3).$
  \item Let $s\in \mathbb{R}, 1\le p,r\le\infty$, and $u \in
\cS'(\R^2).$ Then $u$ belongs to $\dot{B}^{s}_{p, r}(\R^2)$ if and
only if there exists $\{c_{j, r}\}_{j \in \mathbb{Z}} $ such that
$\|c_{j, r}\|_{\ell^{r}} =1$ and
\begin{equation*}
\|\dot{\Delta}_{j}u\|_{L^{p}}\leq C c_{j, r} \, 2^{-j s }
\|u\|_{\dot{B}^{s}_{p, r}}\qquad \mbox{for all}\ \ j\in\Z.
\end{equation*}
\end{enumerate}
\end{rmk}
For the convenience of the reader, in what follows, we recall some
basic facts on Littlewood-Paley theory, one may check \cite{BCD, Tri} for more details.

\begin{lem}\label{lem2.1}
{\sl Let $\cB$ be a ball   and $\cC$ a ring of $\R^2.$
 A constant $C$ exists so
that for any positive real number $\lambda,$ any non negative
integer $k,$ any smooth homogeneous function $\sigma$ of degree $m,$
and any couple of real numbers $(a, \; b)$ with $ b \geq a \geq 1,$
there hold
\begin{equation}
\begin{split}
&\Supp \hat{u} \subset \lambda \mathcal{B} \Rightarrow
\sup_{|\alpha|=k} \|\pa^{\alpha} u\|_{L^{b}} \leq  C^{k+1}
\lambda^{k+ 2(\frac{1}{a}-\frac{1}{b} )}\|u\|_{L^{a}},\\
& \Supp \hat{u} \subset \lambda \mathcal{C} \Rightarrow
C^{-1-k}\lambda^{ k}\|u\|_{L^{a}}\leq
\sup_{|\alpha|=k}\|\partial^{\alpha} u\|_{L^{a}}\leq
C^{1+k}\lambda^{ k}\|u\|_{L^{a}},\\
& \Supp \hat{u} \subset \lambda \mathcal{C} \Rightarrow \|\sigma(D)
u\|_{L^{b}}\leq C_{\sigma, m} \lambda^{ m+2(\frac{1}{a}-\frac{1}{b}
)}\|u\|_{L^{a}}. \end{split}\label{2.1}
\end{equation}}
\end{lem}

In the rest of the paper, we shall frequently use homogeneous Bony's
decomposition \cite{Bony}:
\begin{equation}\label{bony}
uv=T_u v+T'_vu=T_u v+T_v u+\cR(u,v),
\end{equation}
where
\begin{equation*}
\begin{split}
&T_u v\eqdefa\sum_{q \in \mathbb{Z}}\dot S_{q-1}u\dot\Delta_q v,\qquad
T'_vu\eqdefa\sum_{q}\dot\Delta_q u\dot S_{q+2}v,\\
&\cR(u,v)\eqdefa\sum_{q}\dot\Delta_q u \widetilde{\dot\Delta}_{q}v,\quad
\quad \mbox{and}\quad \widetilde{\dot\Delta}_{q}v\eqdefa
\sum_{|q'-q|\leq 1}\dot\Delta_{q'}v.
\end{split}
\end{equation*}

In  order to obtain a better description of the regularizing effect
of the transport-diffusion equation, we will use Chemin-Lerner type
spaces from
\cite{Ch99, CL}.
\begin{defi}\label{chaleur+}
{\sl Let $s\in\R,$
$(r,\lambda,p)\in[1,\,+\infty]^3,$  $T\in]0,\,+\infty]$.
and $u\in{\mathcal S}'(\R^2),$ we set
$$
\|u\|_{\widetilde L^\lambda_T(B^s_{p,r})}\eqdefa\Big(2^{qs}\|\Delta_q
u\|_{L^\lambda_T(L^{p})}\Big)_{\ell ^{r}}
\quad\mbox{and}\quad
\|u\|_{\widetilde L^\lambda_T(\dot B^s_{p,r})}\eqdefa\Big(2^{qs}\|\dot\Delta_q
u\|_{L^\lambda_T(L^{p})}\Big)_{\ell ^{r}},
$$
with the usual modification if $r=\infty.$
\begin{itemize}
\item
For $s\in\R,$ we define
$\widetilde L^\lambda_T(B^s_{p,r})\eqdefa \big\{u\in{\mathcal S}'(\R^2)\;\big|\; \Vert
u\Vert_{\widetilde L^\lambda_T(B^s_{p,r})}<\infty\big\}.$

\item
For $s\le\f2p$ (resp. $s\in\R$), we define
$\widetilde{L}^{\lambda}_T(\dot B^s_{p\,r}(\R^2))$ as the completion
of $C([0,T],\cS(\R^2))$ by norm $\|\cdot\|_{\widetilde L^\lambda_T(\dot B^s_{p,r})}.$
\end{itemize}
In the
particular case when $p=r=2,$ we denote $\widetilde L^\lambda_T(B^s_{2,2})$
(resp. $\widetilde L^\lambda_T(\dot B^s_{2,2})$) by
$\wt{L}^\la_T({H}^s)$ (resp. $\wt{L}^\la_T(\dot{H}^s).$
}
\end{defi}

\begin{rmk}\label{rmk1.2} It is easy to observe that for $\theta\in[0,1],$ we have
\begin{equation}\label{Interpo}
\Vert u\Vert_{\widetilde{L}^{\lambda}_T(\dot B^s_{p,r})} \leq \Vert
u\Vert_{\widetilde{L}^{\lambda_1}_T( \dot B^{s_1}_{p,r})}^{\theta}
\Vert u\Vert_{\widetilde{L}^{\lambda_2}_T(\dot
B^{s_2}_{p,r})}^{1-\theta}
\end{equation}
with
$\frac{1}{\lambda}=\frac{\theta}{\lambda_1}+\frac{1-\theta}{\lambda_2}$
and $s=\theta s_1+(1-\theta)s_2.$ Moreover, Minkowski inequality
implies that
$$
\Vert u\Vert_{\widetilde{L}^{\lambda}_T(\dot B^s_{p,r})} \leq \Vert
u\Vert_{L^{\lambda}_T(\dot B^s_{p,r})}
\quad\mbox{if}\quad\lambda\leq r \quad\hbox{and}\quad \Vert
u\Vert_{L^{\lambda}_T(\dot B^s_{p,r})} \leq \Vert
u\Vert_{\widetilde{L}^{\lambda}_T(\dot B^s_{p,r})}
\quad\mbox{if}\quad r\leq\lambda.
$$
\end{rmk}

\bigbreak \noindent {\bf Acknowledgments.}
G. Gui is supported in part by the National Natural Science Foundation of China under Grants 11571279 and 11331005.

\end{document}